\documentclass[Journal]{IEEEtran}
\IEEEoverridecommandlockouts
\ifCLASSINFOpdf
\else
\fi

\usepackage{mathrsfs}
\usepackage{color}
\usepackage{amsfonts,amssymb}
\usepackage{amsfonts}
\usepackage{subfigure}
\usepackage{booktabs}
\usepackage{graphicx} %插入图片的宏包
\usepackage{float} %设置图片浮动位置的宏包
\usepackage{amsmath}
\usepackage{enumerate}
\usepackage{algorithm}
\usepackage{algpseudocode}
\usepackage{bm}
\usepackage{makecell}
\usepackage{multirow}
\usepackage{epstopdf}
\usepackage{mathtools}
\usepackage{bbm}
\usepackage{dsfont}
\usepackage{pifont}
\usepackage{amsthm}
\usepackage{cite}
\usepackage{balance}
\usepackage[T1]{fontenc}
\usepackage{tikz}

\newtheorem{remark}{Remark}

\newtheorem{theorem}{Theorem}
\newtheorem{lemma}{Lemma}

\newtheorem{assumption}{Assumption}

\newtheorem{definition}{Definition}

\newcommand\tr{\operatorname{tr}}

\newcommand\diag{\operatorname{diag}}

\begin{document}
\title{\LARGE \bf  Near-Optimal Mixed Strategy for Zero-Sum Linear-Quadratic Differential Games}
\author{Tao Xu, Wang Xi, and Jianping He
	\thanks{The authors are with the Department of Automation, Shanghai Jiao Tong University, and Key Laboratory of System Control and Information Processing, Ministry of Education of China, Shanghai 200240, China. E-mail: \{Zerken, bddwyx, jphe\}@sjtu.edu.cn.Preliminary results of this work were presented at IEEE Conference on Decision and Control (CDC), 2023 \cite{xuDifferentialGameMixed2023b}.} }

\maketitle

\begin{abstract}
    Deriving analytic solutions for optimal mixed strategies in zero-sum linear-quadratic differential games (ZSLQDGs) remains an open problem. In this paper, we analytically synthesize near-optimal mixed strategies for ZSLQDGs and establish rigorous performance certifications. Specifically, we construct a surrogate pure-strategy stochastic differential game (SDG) by matching the first two moments of the mixed strategies. This method achieves an $\mathcal{O}(\bar{\pi}^2)$ weak approximation of state distributions and expected costs with respect to the maximum commitment delay $\bar{\pi}$. 
    By analytically resolving the surrogate SDG, we derive closed-form optimal control laws for the matched moments. 
    Crucially, we reveal that the surrogate game is governed by a Generalized Riccati Differential Equation (GRDE), which explicitly dictates a dynamic energy allocation law for variance injection. Building on these solutions, we propose a robust dual-routing architecture to execute the near-optimal mixed strategies. Furthermore, we certify that both the global value approximation error and the strategy suboptimality gaps are bounded by $\mathcal{O}(\bar{\pi}^{\frac{1}{2}})$. Finally, numerical experiments on a double-integrator pursuit-evasion game illustrate the induced physical behaviors and validate the theoretical bounds.
\end{abstract}


\section{Introduction}
Zero-sum differential games, first formalized by Isaacs \cite{isaacsDifferentialGamesMathematical1999}, provide a fundamental framework for modeling antagonistic interactions in continuous time, where one player's gain is exactly the opponent's loss. However, the generality of nonlinear dynamics and cost functionals in zero-sum differential games often leads to intractable Hamilton-Jacobi-Isaacs (HJI) equations. Restricting to linear dynamics and quadratic costs, a zero-sum linear-quadratic differential game (ZSLQDG) admits a remarkable simplification: the game value and optimal pure strategies can be analytically characterized via coupled Riccati differential equations \cite{basarUniquenessNashSolution1976, jacobsonOptimalStochasticLinear2003, weerenAsymptoticAnalysisLinear1999}. This analytic solution not only guarantees global optimality but also facilitates real-time implementation. As a result, ZSLQDGs have become a cornerstone in applications ranging from pursuit--evasion and missile guidance to secure networked control \cite{hoDifferentialGamesOptimal1965a, menonGuidanceLawsSpacecraft1988, turetskyMissileGuidanceLaws2003, liDefendingAssetLinear2011, jagatNonlinearControlSpacecraft2017, aggarwalLinearQuadraticZeroSum2024}.

While classical ZSLQDG theory focuses on pure strategies, where deterministic feedback rules map states to controls, a mixed-strategy formulation naturally extends this paradigm by allowing players to randomize over a family of feedback laws. Mixed strategies have been extensively studied in various fields such as economics \cite{harrisExistenceSubgameperfectEquilibrium1995}, generative adversarial networks \cite{goodfellowGenerativeAdversarialNets2014}, reinforcement learning \cite{perkinsMixedStrategyLearningContinuous2017a}, robotics \cite{petersLearningMixedStrategies2022}, and prominently in pursuit-evasion games \cite{hespanhaProbabilisticPursuitevasionGames2000,vidalProbabilisticPursuitevasionGames2002a, islerRandomizedPursuitevasionPolygonal2005, islerRandomizedPursuitEvasionLocal2006}. In continuous-time games, however, naively defining mixed strategies encounters the fundamental pathology of instantaneous randomization: without a rigorous temporal definition of strategic commitment, each player might revise their randomized choice at the same instant, undermining any meaningful strategic commitment \cite{aumannSpacesMeasurableTransformations1960, aumannBorelStructuresFunction1961, aumann28MixedBehavior1964}. To resolve this, one must explicitly introduce a \textit{commitment delay} into the definition of mixed strategies \cite[p.114]{engwerdaLQDynamicOptimization2005}, ensuring that once a player samples a control action, it remains fixed for a prescribed time interval before the opponent can observe and respond. The commitment delay pattern is therefore essential both for the mathematical well-posedness of mixed strategy designs and for capturing realistic digital control scenarios where strategic interventions cannot be renegotiated infinitely fast.

\subsection{Motivations}
Classical pure-strategy equilibria are fully characterized by the elegant Riccati-equation framework. However, extending this methodology to randomized controls introduces formidable obstacles. Specifically, one must handle infinite-dimensional measure-valued spaces, enforce non-trivial commitment delays, and reconcile continuous-time dynamics with discrete-time probability laws. To date, closed-form analytic descriptions for optimal mixed strategies remain open. It leaves a critical gap in our understanding of how to systematically exploit randomization in continuous-time adversarial control.

Deriving analytic solutions for optimal mixed strategies in ZSLQDGs is essential for advancing the broader theory of differential games. First, closed-form mixed-strategy characterizations would explicitly illuminate the fundamental mechanisms of randomization, clarifying how players dynamically balance their expected control costs against the benefits of adversarial uncertainty. Second, establishing exact analytical benchmarks would dramatically accelerate the development of numerical methods for general mixed-strategy differential games, providing rigorous targets for algorithmic convergence and performance validation.

\subsection{Challenges}
The primary challenge in solving optimal mixed strategies for differential games stems from the inherent intractability of optimizing over infinite-dimensional probability measure spaces. As summarized in Table \ref{tab:game_comparison}, our prior works initiated a weak approximation framework to circumvent this issue. Here, ``weak'' refers to the convergence in distribution of the state-cost processes. We provide a rigorous mathematical treatment of this framework in Section III. However, these methods either lack explicit strategy synthesis \cite{xuDifferentialGameMixed2023b} or rely heavily on numerical control-space discretization \cite{xu2026optimalmixedstrategyzerosum}. Consequently, these existing frameworks inherently yield numerical approximations rather than exact analytic laws, leaving a critical gap for solving ZSLQDGs analytically.

To derive analytic solutions for ZSLQDGs, we pivot to a fundamentally different surrogate construction: parameterizing Markovian mixed strategies via their statistical moments over unbounded spaces. This shift introduces a distinct theoretical hurdle in performance certification. Although the general bounding framework established in \cite{xu2026optimalmixedstrategyzerosum} remains applicable, it provides implicit, term-heavy upper bounds that obscure the exact scaling laws. The crux of the challenge lies in fully exploiting the linear-quadratic structure to derive a refined error analysis, which is essential to collapse these implicit execution penalties into clean, explicit $\mathcal{O}(\bar{\pi}^{1/2})$ error bounds.

\subsection{Contributions}
This paper presents the first closed-form analytical characterization of near-optimal mixed strategies for ZSLQDGs. As contrasted in Table \ref{tab:game_comparison}, our main contributions are three-fold:

\begin{itemize}
    \item \textbf{Markovian formulation and high-order weak approximation:} We construct a surrogate pure-strategy SDG via moment matching. By incorporating state-dependent energy bounds, this formulation natively accommodates unbounded Markovian feedbacks. We prove that this surrogate SDG guarantees the weak approximation accuracy of state distributions and expected costs by $\mathcal{O}(\bar\pi^2)$.
    \item \textbf{Analytic solutions and dynamic energy allocation:} By analytically resolving the surrogate SDG, we bypass intractable HJI partial differential equations. Instead, we reveal that optimal moment trajectories are governed by a Generalized Riccati Differential Equation (GRDE). This uncovers a novel \textit{dynamic energy allocation law} that explicitly dictates variance injection. Based on these closed-form moments, we propose a robust dual-routing architecture for practical strategy execution.
    \item \textbf{Explicit suboptimality certification:} By fully exploiting the linear-quadratic problem structure to design a refined error-decoupling analysis, we successfully replace the implicit implementation penalties of previous works with clean, explicit metrics. We certify that both the value approximation error and the strategy suboptimality gaps are bounded by $\mathcal{O}(\bar\pi^{1/2})$.
\end{itemize}

\begin{table*}[t]
\centering
\footnotesize
\setlength{\tabcolsep}{2pt} 
\renewcommand{\arraystretch}{1.1} % 微调行高，避免表头过挤
\caption{Comparison of game formulations and solution methods across different works}
\label{tab:game_comparison}
\begin{tabular}{l *{11}{c}}
\toprule
& \multicolumn{5}{c}{\textbf{Game Formulation}} 
& \multicolumn{6}{c}{\textbf{Solution Methods}} \\
\cmidrule(lr){2-6} \cmidrule(lr){7-12}
& \multirow{2}{*}[-2ex]{\makecell{Dynamics}} 
& \multirow{2}{*}[-2ex]{\makecell{Control\\Space}} 
& \multirow{2}{*}[-2ex]{\makecell{Cost\\Function}} 
& \multirow{2}{*}[-2ex]{\makecell{Information\\Structure}} 
& \multirow{2}{*}[-2ex]{\makecell{Strategy\\Constraints}} 
& \textbf{(Step 1)} 
& \textbf{(Step 2)} 
& \textbf{(Step 3)} 
& \textbf{(Step 4)} 
& \multicolumn{2}{c}{\textbf{(Step 5)}} \\
\cmidrule(lr){7-10} \cmidrule(lr){11-12}
& & & & & 
& \makecell{Design\\$\tilde{G}$} 
& \makecell{Weak\\Approx.} 
& \makecell{Solve\\$\tilde{G}$} 
& \makecell{Strategy\\Synthesis}
& \makecell{Value\\Approx.}
& \makecell{Suboptimality\\Gap} \\
\midrule
\makecell{This \\ work}
& linear & unbounded & quadratic & \makecell{Markovian\\feedback} & \makecell{state-dep.\\2nd-moment\\ bound} 
& \makecell{moment\\matching} & $\mathcal{O}(\bar{\pi}^2)$ & GRDE & ZOH & $\mathcal{O}(\bar{\pi}^{\frac{1}{2}})$ & $\mathcal{O}(\bar{\pi}^{\frac{1}{2}})$ \\
\midrule
\cite{xu2026optimalmixedstrategyzerosum} 
& bounded & compact & bounded & NAD & $\times$ 
& \makecell{control\\space\\discretization} & $\mathcal{O}(\bar{\pi})$ & HJI & ZOH & \makecell{$\max\{\epsilon_1,\epsilon_2\}$\\$+\mathcal{O}(\bar{\pi})$} & \makecell{$\max\{\epsilon_4,\epsilon_5\}$ \\ $+\epsilon_3+\mathcal{O}(\bar{\pi})$} \\
\midrule
\cite{xuDifferentialGameMixed2023b} 
& linear & unbounded & \makecell{polynomial\\growth} & open-loop & $\times$ 
& \makecell{moment\\matching} & $\mathcal{O}(\bar{\pi})$ & HJI & $\times$ & $\times$ & $\times$ \\
\bottomrule
\end{tabular}
\end{table*}

\subsection{Related Works}
The instantaneous-randomization issue in continuous-time mixed strategies was raised in the seminal work \cite{aumann28MixedBehavior1964}. Subsequent research on mixed strategies in differential games has proceeded along two complementary lines, distinguished by how controls are modeled. The first line retains real-valued feedback controls and enforces commitment delays in the strategy definition to prevent pathological instantaneous randomization \cite{cardaliaguetDifferentialGamesAsymmetric2007, buckdahnValueFunctionDifferential2013, cardaliaguetPureRandomStrategies2014,buckdahnValueMixedStrategies2014, flemingMixedStrategiesDeterministic2017a, xu2026optimalmixedstrategyzerosum}. In these works, commitment delays are typically assumed to vanish asymptotically to recover the classical game value, which is characterized as the unique viscosity solution of an HJI equation. The second line builds on the relaxed-control framework \cite{berkovitzRelaxedControls2012}, treating admissible controls directly as measure-valued functions and leveraging martingale or weak-formulation techniques to define equilibria \cite{sirbuStochasticPerronsMethod2014,kunischOptimalControlUndamped2016,coxControlledMeasurevaluedMartingales2024}. In this paper, we adopt the real-valued control perspective with explicit, non-vanishing commitment delays, aligning with the first line of research to preserve direct physical implementability.

Despite the solid theoretical foundation provided by these studies, analytic solutions for optimal mixed strategies in ZSDGs remain critically limited. Prior efforts have focused primarily on proving the existence of, and characterizing, the value function under vanishing delays, rather than yielding explicit strategy laws \cite{cardaliaguetDifferentialGamesAsymmetric2007, buckdahnValueFunctionDifferential2013, cardaliaguetPureRandomStrategies2014,buckdahnValueMixedStrategies2014}. First, solving the associated HJI equations in infinite-dimensional measure spaces is notoriously intractable: even pure-strategy solutions lack closed-form expressions except in very special cases \cite{elliottExistenceValueStochastic1976}, and classical methods such as the Pontryagin maximum principle \cite{wangPontryaginsMaximumPrinciple2010} or viscosity-solution techniques \cite{flemingControlledMarkovProcesses2006} do not readily generalize. Second, only a handful of studies have tackled specific instances: Kumar's generalized bomber-and-battleship game yields an explicit mixed-strategy solution \cite{kumarOptimalMixedStrategies1980}, and Fleming and Hernández established the existence of approximately optimal strategies \cite{flemingMixedStrategiesDeterministic2017a}. While more recent works often resort to neural-network-based policy learning \cite{dou2019finding, martinFindingMixedstrategyEquilibria2023}, these approaches typically sacrifice analytical transparency. Third, forcing commitment delays toward zero induces chattering controls that switch infinitely fast, undermining both physical implementability and system robustness \cite{levantChatteringAnalysis2010}.

More recently, our prior works initiated an SDG weak approximation framework to circumvent the infinite-dimensional measure optimization. While an open-loop moment-matching surrogate was explored in \cite{xuDifferentialGameMixed2023b}, a systematic near-optimal design was later established in \cite{xu2026optimalmixedstrategyzerosum} utilizing control-space discretization. Although the latter provides rigorous $\mathcal{O}(\bar\pi)$ performance guarantees for general nonlinear ZSDGs, its structural reliance on spatial discretization fundamentally restricts the output to numerical approximations. This methodological gap directly motivates our current work. By specializing the weak approximation framework to ZSLQDGs and parameterizing Markovian mixed strategies via their statistical moments, we completely bypass numerical spatial discretization. This paradigm shift enables us to derive exact analytic mixed strategies governed by GRDEs, elevate the weak approximation accuracy to $\mathcal{O}(\bar\pi^2)$, and rigorously certify the suboptimality gaps, ultimately yielding closed-form strategies that are directly implementable in digital systems.

\textbf{Organization.} The remainder of this paper is organized as follows. Section II formally formulates the ZSLQDG problem under mixed strategies. Section III introduces the SDG weak approximation framework and proves that the moment-matching surrogate game achieves a second-order approximation accuracy. In Section IV, we analytically resolve the surrogate game to synthesize the near-optimal mixed strategies and rigorously certify the bounds for both the value approximation error and the strategy suboptimality gaps. Section V validates our theoretical results and explores the physical behaviors of mixed strategies via a 2-D double-integrator pursuit-evasion game. Finally, Section VI concludes the paper.

\textbf{Notation.} Let $\pi = \{t_{0} < t_{1} < \ldots < t_{N} = T\}$ with a finite positive integer $N$ be a partition of a time interval $[t_0, T]$, and $\Pi_{[t_0,T]}$ be the set of all partitions on $[t_0,T]$. We define the lower bound, upper bound, and order of $\pi$ as follows:
\begin{equation*}
    \underline{\pi} \triangleq \min_{0 \leq k \leq N-1}(t_{k+1}-t_{k}),\; \bar{\pi} \triangleq \max_{0 \leq k \leq N-1}(t_{k+1}-t_{k}),\;|\pi|\triangleq N.
\end{equation*} 
A partition $\pi$ is called equispaced if $\underline{\pi} = \bar{\pi}$. 
The $k$-th time interval, $[t_{k-1}, t_k)$ is denoted as $\Delta_k$, with length $|\Delta_k|\triangleq t_k-t_{k-1}$. The indicator function $\mathbb{I}_{\Delta_k}(t)$ equals $1$ if $t\in \Delta_k$, otherwise equals $0$. 
$\zeta_\pi:[t_0,T]\to\{1,\ldots,|\pi|\}$ is the index function of $\pi$, and $\delta_\pi:[t_0,T]\to[\underline{\pi},\bar{\pi}]$ is the interval length function of $\pi$, i.e., 
\[\zeta_{\pi}(t) \triangleq \sum_{k=1}^{|\pi|} k\,\mathbb{I}_{\Delta_k}(t), \quad \delta_{\pi}(t)\triangleq\sum_{k=1}^{|\pi|}|\Delta_k|\,\mathbb{I}_{\Delta_k}(t).\] 
For a $U$-valued random variable $x$ and a map $h:U\to \mathbb{R}^d$, the expectation and covariance of $h(x)$ are denoted as $\mathbb{E}[h(x)]$ and $\mathbb{D}[h(x)]$, respectively. Let $\mathcal{P}(U)$ denote the set of probability measures on $U$. We denote the set of positive semi-definite matrices of dimension $n$ as $S_{+}^{n}$. Given $A, B\in\mathbb{R}^{n\times n}$, we use $A \succeq B$ to indicate $A-B \in S_{+}^{n}$.

\section{Problem Formulation}
In this section, we first present the system dynamics and the associated quadratic cost functional. Then, we define the admissible mixed strategies as conditionally independent Markovian randomized feedbacks and characterize the state-dependent energy constraints. Finally, we formalize the game value and outline the main problems of interest.

\subsection{Dynamics and Cost} 
Consider a continuous-time ZSLQDG defined on the time horizon $[t_0, T]$. In contrast to the classical pure-strategy paradigm where control inputs are deterministic functions, we consider a broader class of games where players may randomize their actions. Consequently, the control inputs $u(t)$ and $v(t)$ are generally modeled as stochastic processes. The dynamics and cost of the game are as follows:
\begin{equation*}\label{eq:G_lq}
(\mathbf{G}_{lq})\!:\!\left\{\begin{aligned}
&\dot{x}(t) = Ax(t) \!+\! B_{1}u(t) \!+\! B_{2}v(t),\quad x(t_0)= x_0,\\
&J(t_0,x_0,u,v) = \mathbb{E}\bigg\{\int_{t_0}^T\!\!\big[x^{\top}(t)Qx(t)+u^{\top}(t)R_1u(t)\\
&\qquad\qquad\quad - v^{\top}(t)R_2v(t)\big] dt + x^{\top}(T)Q_Tx(T)\bigg\},
\end{aligned}\right.
\end{equation*}
where $x(t) \in \mathbb{R}^d$ is the state vector. Player 1 (the minimizer) and Player 2 (the maximizer) take their control inputs in unbounded continuous spaces $U = \mathbb{R}^{d_1}$ and $V = \mathbb{R}^{d_2}$, respectively. The expectation operator $\mathbb{E}[\cdot]$ encapsulates the stochasticity intentionally injected via the players' randomized mixed strategies, whose rigorous definition will be formalized in Section II-B.

\begin{assumption}[System matrices]\label{ass:system_matrices}
The weighting matrices $Q$, $Q_T$, and $R_i$ for $i=1,2$, are positive definite. Additionally, the pair $(A,B_1)$ is stabilizable, and the pair $(A, Q^{\frac{1}{2}})$ is detectable.
\end{assumption}

For an individual realized trajectory, we define the accumulated cost function evaluated at time $t\in[t_0,T]$ as
\begin{equation}\label{eq:accumulated_cost}
\operatorname{cost}(t)=\left\{\begin{array}{ll}
\int_{t_0}^t h(s,x,u,v)ds,  & t < T, \\ 
\int_{t_0}^t h(s,x,u,v)ds + g(x(T)), & t=T,
\end{array}\right.
\end{equation}
where $h(s,x,u,v) = x^\top(s)Qx(s) + u^\top(s)R_1u(s)-v^\top(s)R_2v(s)$ and $g(x(T)) = x^\top(T) Q_T x(T)$. By definition, $\mathbb{E}[\operatorname{cost}(T)] = J(t_0,x_0,u,v)$.

\subsection{Admissible Mixed Strategy}
In realistic digital control systems, it is standard for players to update their actions at discrete sampling instants due to communication latencies and computational constraints. To align with this practical mechanism, we introduce the commitment delay pattern into the continuous-time game.

\begin{assumption}[Information structure and commitment delay]\label{ass:commit}
The control processes are governed by a commitment delay pattern $\pi = \{t_0 < t_1 < \dots < t_N = T\}$. During each sampling interval $[t_{k}, t_{k+1})$ for $k \in \{0, \dots, N-1\}$, both players must commit to a constant control action determined at the sampling instant $t_{k}$ based on the available state $x(t_{k})$, i.e., $u(t) = u_{k}$ and $v(t) = v_{k}$ for $t \in [t_{k}, t_{k+1})$.
\end{assumption}
While this assumption on shared commitment delay pattern is standard in the literature of sampled-data differential games \cite{buckdahnValueFunctionDifferential2013,buckdahnValueMixedStrategies2014}, we acknowledge that in practical decentralized scenarios, players may operate on different temporal grids. Analyzing such games with asynchronous delays, though beyond the scope of this paper, remains an important direction for future research.

Furthermore, since the control spaces $U$ and $V$ are unbounded, unrestricted randomization could yield infinite expected costs and destroy the well-posedness of the LQ framework. Therefore, it is necessary to physically bound the randomization. 
In linear feedback systems, the nominal control effort naturally scales linearly with the state deviation \cite{anderson2007optimal}. Concurrently, in networked control environments under Signal-to-Noise Ratio (SNR) constraints, the variance of injected control noise is fundamentally proportional to the signal power (i.e., state magnitude) \cite{braslavsky2007feedback, wonham1967optimal}. Motivated by these physical rationales, we impose state-dependent second-moment constraints on the probability measures, ensuring that the total injected stochastic energy dynamically scales with the system's current deviation.

\begin{definition}[Admissible mixed strategy]\label{def:admissible_mixed}
An admissible mixed strategy for Player 1, denoted as $\alpha \in \mathcal{A}_m^\pi$, is a sequence of measurable Markovian maps $\{\mu_k\}_{k=0}^{N-1}$ such that:
\begin{equation}
\mu_k : \mathbb{R}^d \to \mathcal{P}(U), \quad x(t_{k}) \mapsto \mu_k(\cdot \mid x(t_{k})).
\end{equation}
At each sampling instant $t_{k}$, Player 1 samples an action $u_{k} \sim \mu_k(\cdot \mid x(t_{k}))$. Additionally, the probability measures are subject to the state-dependent second-moment constraint:
\begin{equation}\label{eq:energy_constraint_1}
\int_{U} \|u\|^2 \mu_k(du \mid x(t_{k})) \le \gamma_1^2 \|x(t_{k})\|^2,
\end{equation}
where $\gamma_1 > 0$ represents the actuation capacity coefficient. 

Symmetrically, the admissible strategy space $\mathcal{B}_m^\pi$ for Player 2 is defined as a sequence of measurable Markovian maps $\{\nu_k\}_{k=0}^{N-1}$ subject to the actuation capacity $\gamma_2 > 0$:
\begin{equation}\label{eq:energy_constraint_2}
\int_{V} \|v\|^2 \nu_k(dv \mid x(t_{k})) \le \gamma_2^2 \|x(t_{k})\|^2.
\end{equation}
\end{definition}

To preserve fairness and non-anticipativity in this simultaneous-move game, the intrinsic randomization mechanisms (i.e., random seeds) employed by the players must remain strictly private and causally decoupled. This prevents any player from anticipatively observing or exploiting the opponent's action within the same commitment interval. 
\begin{assumption}[Conditional independence]\label{ass:independence}
Given the commonly observed state $x(t_{k})$, the random controls $u_{k}$ and $v_{k}$ are sampled independently, i.e., the joint probability measure factorizes as $\mathbb{P}(u_{k}, v_{k} \mid x(t_k)) = \mu_k(u_{k} \mid x(t_k))\nu_k(v_{k} \mid x(t_k))$.
\end{assumption}

\begin{remark}[Markovian feedback vs. NAD formulation]\label{rem:markov_vs_nad}
For general ZSDGs with bounded control spaces, existing literature \cite{buckdahnValueFunctionDifferential2013, buckdahnValueMixedStrategies2014, xu2026optimalmixedstrategyzerosum} traditionally defines mixed strategies through the framework of non-anticipative strategies with delay (NAD). A core motivation for adopting the NAD functional mapping is to rigorously characterize the game value as the unique viscosity solution to the associated HJI equation. In contrast, for the ZSLQDG addressed in this paper, our weak approximation framework allows the surrogate game to be solved analytically via a GRDE. Since the GRDE admits explicit classical solutions, the NAD framework and viscosity solutions are no longer necessitated, which directly yields an easily implementable Markovian control architecture.
\end{remark}

\subsection{Problems of Interest}
The sequential execution of any admissible mixed strategy pair $(\alpha, \beta) \in \mathcal{A}_m^\pi \times \mathcal{B}_m^\pi$, coupled with the initial state $x_0$, uniquely induces a joint probability measure over the sample paths of the control processes $u$ and $v$. Consequently, the expected cost under a specific strategy pair is unambiguously defined as $J(t_0, x_0, \alpha, \beta)$, which evaluates the cost functional under this induced measure.

To investigate the equilibrium, we define the upper and lower value functions as:
\begin{equation}\label{eq:upper_lower_value}
\begin{aligned}
V_{m,+}^\pi(t_0,x_0) &\triangleq \inf_{\alpha \in \mathcal{A}_m^\pi} \sup_{\beta \in \mathcal{B}_m^\pi} J(t_0,x_0,\alpha,\beta), \\
V_{m,-}^\pi(t_0,x_0) &\triangleq \sup_{\beta \in \mathcal{B}_m^\pi} \inf_{\alpha \in \mathcal{A}_m^\pi} J(t_0,x_0,\alpha,\beta).
\end{aligned}
\end{equation}

We are interested in the following three core problems:
\begin{itemize}
\item \textbf{Existence of game value.} Establish whether the saddle-point equilibrium exists for $\mathbf{G}_{lq}$ under the defined spaces $\mathcal{A}_m^\pi \times \mathcal{B}_m^\pi$, i.e., $V_{m,+}^\pi(t_0,x_0) = V_{m,-}^\pi(t_0,x_0)$.
\item \textbf{State and cost characterization.} Given any pair of mixed strategies, rigorously characterize the continuous-time evolution of the system state distribution and the accumulation of expected costs over the horizon.
\item \textbf{Near-optimal strategy and certification.} Derive analytic mixed strategies by solving a surrogate pure-strategy SDG via GRDE, and certify both the value approximation error and the strategy suboptimality gap.
\end{itemize}

To address these problems, our main idea is to weakly approximate the original mixed-strategy game $\mathbf{G}_{lq}$ by a pure-strategy SDG $\mathbf{\tilde{G}}_{lq}$, then derive both approximated game value and near-optimal mixed strategy with performance certifications. The key points are illustrated in Fig. \ref{fig:roadmap}.
\begin{figure}[h]
    \centering
    \includegraphics[width=0.8\linewidth]{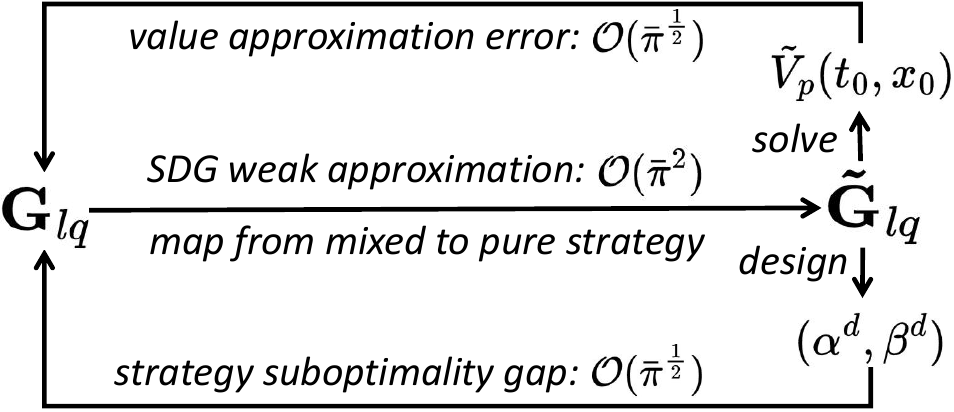}
    \caption{Illustration of the main ideas.}
    \label{fig:roadmap}
\end{figure}

\section{Game Value Existence and SDG Weak Approximation}
In this section, we address the first two problems of interest outlined in Section II. First, we establish the mathematical well-posedness of the game by proving the existence of a unique saddle-point game value. Second, to analytically characterize the evolution of the system state and expected costs driven by the players' discrete-time randomizations, we introduce a weak approximation framework. By mapping the random interventions to continuous-time diffusion processes, we construct a mathematically tractable surrogate SDG.

\subsection{Existence of Game Value}
Before constructing the surrogate approximation, it is imperative to ensure that the original game $\mathbf{G}_{lq}$ is well-posed under the defined mixed strategy spaces. Benefiting from the ZOH commitment mechanism and the weakly compact probability measure spaces induced by the state-dependent energy constraints, the existence of the game value is guaranteed.

\begin{theorem}[Existence of game value]\label{thm:SE}
The linear-quadratic game $\mathbf{G}_{lq}$ under the admissible Markovian mixed strategy spaces $\mathcal{A}_m^\pi\times\mathcal{B}_m^\pi$ has a unique game value, denoted as $V^\pi_m(t_0,x_0)$.
\end{theorem}
\begin{proof}
   We establish the existence of the game value via backward induction, starting from the terminal time $T$. For each interval $[t_{k-1}, t_k)$, we demonstrate the existence of a local saddle point by invoking Sion's Minimax Theorem. We rigorously verify its requisite conditions by showing that: (i) the energy-constrained mixed strategy spaces are weakly compact, as ensured by the uniform second-moment bounds and Prokhorov’s Theorem; and (ii) the local expected cost functional is bilinear and continuous in the weak topology. The recursive coincidence of the local upper and lower values yields the global game value. Full details are provided in Appendix \ref{app:thm:SE}.
\end{proof}

\subsection{Weak Approximation Framework}
Characterizing the continuous-time evolution of the state distribution and expected cost under mixed strategies (our second problem of interest) is analytically intractable. The piecewise-constant random interventions inherently create a complex stochastic process whose exact pathwise behavior cannot be evaluated in closed form. However, since the cost functional of $\mathbf{G}_{lq}$ relies solely on the \textit{expectations} of these trajectories, exact pathwise matching (i.e., strong approximation) is unnecessary. Instead, it is sufficient to capture the macroscopic probability distributions over time.

Let $\operatorname{PG}$ denote the set of continuous functions $\mathbb{R}^d \rightarrow \mathbb{R}$ of at most polynomial growth, i.e., $\psi(\cdot) \in \operatorname{PG}$ if there exist $\kappa_1, \kappa_2>0$ such that $|\psi(x)| \leq \kappa_1\left(1+|x|^{\kappa_2}\right)$ for all $x \in \mathbb{R}^d$. If $\kappa_2=1$, $\psi$ is of linear growth. Moreover, $\operatorname{PG}^n$ denotes the set of $n$-times continuously differentiable functions whose partial derivatives up to order $n$ belong to $\operatorname{PG}$.
Given a partition $\pi$ on $[t_0,T]$, a continuous-time process $\{\tilde{x}(t)\}_{t\in[t_0,T]}$ and a discrete-time process $\{x(k)\}_{k=0}^{|\pi|}$ weakly approximate each other if their probability distributions at the discrete time steps $t_k$ are uniformly close.

\begin{definition}[Weak approximation]\label{def:weak_approximation}
    Given a partition $\pi$ on $[t_0,T]$, a continuous-time stochastic process $\left\{\tilde{x}(t)\right\}_{t \in[t_0, T]}$ is an order $n$ weak approximation of a discrete-time stochastic process $\left\{x(k)\right\}_{k=0}^{|\pi|}$ if, for any polynomial-growth function $\psi \in \operatorname{PG}^{n+1}$, there is a positive constant $C$ independent of $\bar\pi$ such that:
    \begin{equation*}
        \max _{k=0, \ldots, |\pi|}\left|\mathbb{E} \psi(x(k))-\mathbb{E} \psi(\tilde{x}(t_k))\right| \leq C \bar{\pi}^n.
    \end{equation*}
\end{definition}

This concept is generalized to the game setting in \cite{xu2026optimalmixedstrategyzerosum}. A surrogate SDG under pure strategies weakly approximates the original ZSDG under mixed strategies if both the state distributions and the expected accumulated costs are well approximated.
The approximation error bounds defined in terms of $\mathcal{O}(\bar{\pi}^n)$ are asymptotic properties valid in the limit $\bar{\pi} \to 0$. In the context of differential games, this corresponds to high-frequency sampling where the commitment delay is sufficiently small. Within this regime ($\bar{\pi} < 1$), a higher order $n$ provides a strictly superior approximation accuracy.

\begin{definition}[SDG weak approximation \cite{xu2026optimalmixedstrategyzerosum}]\label{def:SDG_weak_approximation}
    A zero-sum SDG $(\mathbf{\tilde{G}})$ under pure strategy spaces $\tilde{\mathcal{A}}_p^\pi\times\tilde{\mathcal{B}}_p^\pi$ is an order $n$ weak approximation of a zero-sum differential game $(\mathbf{G})$ under mixed strategy spaces $\mathcal{A}_m^\pi\times\mathcal{B}_m^\pi$ if: 
    \begin{enumerate}
        \item[(i)] there exist 
    \begin{equation}\label{eq:strategy-map}
        \phi_1: \mathcal{A}_m^\pi \to\tilde{\mathcal{A}}_p^\pi\text{ and }\phi_2: \mathcal{B}_m^\pi \to\tilde{\mathcal{B}}_p^\pi,
    \end{equation}
    where for any $(\alpha, \beta)\in\mathcal{A}_m^\pi\times\mathcal{B}_m^\pi$, one has $\tilde{\alpha}=\phi_1(\alpha)\in\tilde{\mathcal{A}}_p^\pi$ and $\tilde{\beta}=\phi_2(\beta)\in\tilde{\mathcal{B}}_p^\pi$. 
    \item[(ii)] the state trajectory of the SDG $\mathbf{\tilde{G}}$ under strategies $(\tilde{\alpha}, \tilde{\beta})$, i.e., $\{\tilde{x}(t)\}_{t\in[t_0,T]}$, is an order $n$ weak approximation of the state trajectory of the game $\mathbf{G}$ under strategies $(\alpha, \beta)$ on the time steps of $\pi$, i.e., $\{x(t_k)\}_{k=0}^{|\pi|}$.
    \item[(iii)] the accumulated cost trajectory of the SDG $\mathbf{\tilde{G}}$ under strategies $(\tilde{\alpha}, \tilde{\beta})$ approximates the accumulated cost trajectory of the game $\mathbf{G}$ under strategies $(\alpha, \beta)$ in an expected sense that $\left|\mathbb{E}\left[\operatorname{cost}(t_k)\right] - \mathbb{E}\left[\tilde{\operatorname{cost}}(t_k)\right]\right| = \mathcal{O}(\bar{\pi}^n)$ for $k=0,\ldots, |\pi|$.
    \end{enumerate}
\end{definition}
Unlike classical weak approximation for SDEs, which typically concerns the distribution of the state process $x_t$, our definition enforces the joint weak convergence of the state process and the accumulated cost process $(x_t, \text{cost}_t)$. This joint convergence is essential for differential games because the value function is defined via the cost functional. Convergence of the state alone is insufficient to guarantee the convergence of the game value. By matching both the state and the cost process, our framework ensures that the saddle point of the surrogate SDG provides a mathematically rigorous approximation of the original game value.

\subsection{Design of the Surrogate SDG $\mathbf{\tilde{G}}_{lq}$}
The intuition behind our moment-matching design is that, instead of choosing a distribution $\mu_k$, Player 1 in the surrogate game directly controls the conditional expectation and conditional covariance of their actions, effectively shaping the drift and diffusion of the system.

Formally, let $\{W_t\}_{t\in[t_0,T]}$ be a standard $(d_1+d_2)$-dimensional Wiener process defined on a probability space $(\tilde{\Omega},\tilde{\mathcal{F}},\tilde{P})$ with filtration $\tilde{\mathbb{F}}$. 
We define the virtual pure controls for the surrogate game as the tuple processes $\tilde{u}(t) = (\tilde{\Gamma}_1(t),\tilde{\Lambda}_1(t))$ and $\tilde{v}(t) = (\tilde{\Gamma}_2(t),\tilde{\Lambda}_2(t))$, where $\tilde{\Gamma}_i(t)$ takes values in $\mathbb{R}^{d_i}$ and $\tilde{\Lambda}_i(t)$ takes values in $\mathbb{R}^{d_i \times d_i}$. The covariance matrix is defined as $\tilde{\Sigma}_i(t) \triangleq \tilde{\Lambda}_i(t)\tilde{\Lambda}_i^{\top}(t) \succeq \mathbf{0}$.

Driven by these moment-based controls, the continuous-time surrogate ZSLQDG is constructed as follows:
\begin{equation*}
    (\mathbf{\tilde{G}}_{lq})\!:\!\left\{\begin{aligned}
        &d \tilde{x}(t) = \left(A\tilde{x}(t) + B_1\tilde{\Gamma}_1(t) + B_2\tilde{\Gamma}_2(t) \right)d t \\
        & \qquad+\delta_{\pi}^{\frac{1}{2}}(t)\left[B_1\tilde{\Lambda}_1(t), B_2\tilde{\Lambda}_2(t)\right]d W_t, \; \tilde{x}(t_0) = x_0,\\
        & \tilde{J}(t_0,x_0,\tilde{u},\tilde{v}) = \mathbb{E}\!\int_{t_0}^T\!\left[\tilde{x}^{\top}(t)Q\tilde{x}(t) \!+\! \tilde{\Gamma}_1^{\top}(t)R_1\tilde{\Gamma}_1(t)\right.\\
        &\quad\qquad- \tilde{\Gamma}_2^{\top}(t)R_2\tilde{\Gamma}_2(t) + \tr(\tilde{\Lambda}_1(t)^{\top}R_1\tilde{\Lambda}_1(t)) \\
        &\quad\qquad\left.- \tr(\tilde{\Lambda}_2(t)^{\top}R_2\tilde{\Lambda}_2(t))\right] dt +\tilde{x}^{\top}(T)Q_T\tilde{x}(T).
    \end{aligned}\right.
\end{equation*}

A fundamental advantage of this formulation emerges when translating the state-dependent energy constraints. Using the well-known variance identity $\mathbb{E}[\|u\|^2] = \|\mathbb{E}[u]\|^2 + \operatorname{tr}(\mathbb{D}[u])$, the original integral constraints over probability measures gracefully translate into deterministic algebraic boundaries. This enables us to directly define the admissible strategies for the surrogate game symmetrically to the original game.

\begin{definition}[Admissible pure strategy]\label{def:tilde_admissible}
    An admissible pure strategy for Player 1 in the surrogate game, denoted as $\tilde{\alpha} \in \tilde{\mathcal{A}}_p^\pi$, is a sequence of measurable Markovian maps $\{\tilde{\mu}_k\}_{k=0}^{N-1}$:
    \begin{equation}
        \tilde{\mu}_k : \mathbb{R}^d \to \mathbb{R}^{d_1} \times \mathbb{R}^{d_1 \times d_1}, \quad \tilde{x}(t_k) \mapsto (\tilde{\Gamma}_1(t_k), \tilde{\Lambda}_1(t_k)).
    \end{equation}
    At each sampling instant $t_k$, the extracted moments are held constant for $t \in [t_k, t_{k+1})$ and must satisfy the state-dependent algebraic boundary:
    \begin{equation}\label{eq:tilde_energy_constraint_1}
        \|\tilde{\Gamma}_1(t_k)\|^2 + \tr(\tilde{\Sigma}_1(t_k)) \le \gamma_1^2 \|\tilde{x}(t_k)\|^2,
    \end{equation}
    where $\tilde{\Sigma}_1(t_k) \triangleq \tilde{\Lambda}_1(t_k)\tilde{\Lambda}_1^\top(t_k) \succeq \mathbf{0}$. 
    
    Symmetrically, the admissible pure strategy space $\tilde{\mathcal{B}}_p^\pi$ for Player 2 is defined as a sequence of maps $\{\tilde{\nu}_k\}_{k=0}^{N-1}$ mapping $\tilde{x}(t_k)$ to $(\tilde{\Gamma}_2(t_k), \tilde{\Lambda}_2(t_k))$, subject to $\|\tilde{\Gamma}_2(t_k)\|^2 + \tr(\tilde{\Sigma}_2(t_k)) \le \gamma_2^2 \|\tilde{x}(t_k)\|^2$.
\end{definition}

By replacing the infinite-dimensional measure spaces with Euclidean moment matrices, the original game is successfully relaxed into a highly tractable pure-strategy SDG that strictly preserves the LQ structure. Let $\tilde{V}_{p,+}^\pi$ and $\tilde{V}_{p,-}^\pi$ denote its upper and lower value functions, respectively.

\subsection{Strategy Projection and Certification of Order $2$}
To mathematically validate that $\mathbf{\tilde{G}}_{lq}$ satisfies the weak approximation requirements (Definition \ref{def:SDG_weak_approximation}), we must construct explicit strategy projection maps $(\phi_1, \phi_2)$ that connect the two distinct probability spaces.

\begin{definition}[Strategy projection maps]\label{def:map_lq}
    For Player 1 employing an admissible Markovian mixed strategy $\alpha = \{\mu_k\}_{k=0}^{N-1} \in \mathcal{A}_m^\pi$, the projection map $\phi_1$ constructs a surrogate pure strategy $\tilde{\alpha} = \{\tilde{\mu}_k\}_{k=0}^{N-1} \in \tilde{\mathcal{A}}_p^\pi$ by extracting the conditional statistical moments. Specifically, the surrogate controls $(\tilde{\Gamma}_1(t_k), \tilde{\Lambda}_1(t_k))$ evaluated at the mapped state $\tilde{x}(t_k)$ are defined by:
    \begin{equation*}
        \begin{aligned}
            \tilde{\Gamma}_1(t_k) &= \int_{U} u \, \mu_k(du \mid \tilde{x}(t_k)), \\
            \tilde{\Sigma}_1(t_k) &= \int_{U} \big(u - \tilde{\Gamma}_1(t_k)\big)\big(u - \tilde{\Gamma}_1(t_k)\big)^\top \mu_k(du \mid \tilde{x}(t_k)),
        \end{aligned}
    \end{equation*} 
    where $\tilde{\Sigma}_1(t_k) \triangleq \tilde{\Lambda}_1(t_k)\tilde{\Lambda}_1^\top(t_k) \succeq \mathbf{0}$. The map $\phi_2$ for Player 2 is symmetrically defined via $\nu_k$. 
\end{definition}
\begin{remark}[Feasibility]
    By the variance identity $\int \|u\|^2 d\mu = \|\mathbb{E}[u]\|^2 + \tr(\mathbb{D}[u])$, the energy constraint of $\mathbf{G}_{lq}$ intrinsically ensures $\|\tilde{\Gamma}_i(t_k)\|^2 + \tr(\tilde{\Sigma}_i(t_k)) \le \gamma_i^2 \|\tilde{x}(t_k)\|^2$. Thus, the maps $(\phi_1, \phi_2)$ are strictly feasible for $\mathbf{\tilde{G}}_{lq}$.
\end{remark}

With the projection maps rigorously defined, we establish the approximation performance. By matching both the first and second moments, the surrogate game accurately traces the original game dynamics.

\begin{theorem}[$\mathbf{\tilde{G}}_{lq}$ approximation]\label{thm:g_lq_weak_approximation}
    Under the designed projection maps $(\phi_1, \phi_2)$, the surrogate pure-strategy SDG $\mathbf{\tilde{G}}_{lq}$ is an order $2$ weak approximation (i.e., $\mathcal{O}(\bar\pi^2)$) of the original mixed-strategy game $\mathbf{G}_{lq}$.
\end{theorem}
\begin{proof}
    We certify the $\mathcal{O}(\bar{\pi}^2)$ weak approximation using Milshtein’s framework. First, we augment the state dynamics with the accumulated cost variable to enable joint state-cost approximation. Second, we perform a local one-step error analysis between the exact game increment and the surrogate SDG, proving that their first three moments deviate by at most $\mathcal{O}(\bar{\pi}^3)$. Finally, by aggregating these local deviations over $N$ sampling intervals, we rigorously establish that the global approximation error is bounded by $\mathcal{O}(\bar{\pi}^2)$. Full details are provided in Appendix \ref{app:thm:g_lq_weak_approximation}.
\end{proof}

In our prior works on general nonlinear ZSDGs \cite{xu2026optimalmixedstrategyzerosum}, the weak approximation error was strictly bounded by $\mathcal{O}(\bar{\pi})$ (order $1$). Remarkably, owing to the global linearity of the dynamics and the uncentered moment substitutions in the ZSLQDG framework, the local truncation errors of the third-order moments vanish structurally. This elevates the global weak approximation accuracy to $\mathcal{O}(\bar{\pi}^2)$ (order $2$), providing a significantly tighter characterization of both state distributions and expected costs.

\section{Near-Optimal Mixed Strategies}
In this section, we analytically resolve the surrogate SDG $\mathbf{\tilde{G}}_{lq}$ to extract closed-form optimal pure strategies. By inversely mapping these continuous-time deterministic moments back to the discrete-time probability spaces, we synthesize near-optimal mixed strategies for the original game $\mathbf{G}_{lq}$. Finally, we develop a rigorous performance certification result, systematically bounding the global value approximation error and the strategy suboptimality gaps.

\subsection{Solve $\mathbf{\tilde{G}}_{lq}$}
Solving the surrogate game $\mathbf{\tilde{G}}_{lq}$ directly under the delay-dependent strategy spaces $\tilde{\mathcal{A}}_p^\pi \times \tilde{\mathcal{B}}_p^\pi$ is analytically challenging. To derive closed-form optimal control laws, we analyze the game under the asymptotic continuous-time regime where the commitment delays approach zero, yielding the delay-free admissible spaces $\tilde{\mathcal{A}}_p \triangleq \cup_{\pi}\tilde{\mathcal{A}}_p^\pi$ and $\tilde{\mathcal{B}}_p \triangleq \cup_{\pi}\tilde{\mathcal{B}}_p^\pi$. Since the surrogate linear-quadratic SDG rigorously satisfies the Isaacs condition, its saddle-point game value, denoted as $\tilde{V}_p(t_0,x_0)$, is well-defined over $\tilde{\mathcal{A}}_p\times\tilde{\mathcal{B}}_p$. This condition holds naturally because the Hamiltonian is separable in $\tilde{u}$ and $\tilde{v}$, implying that the minimax equality holds \cite{basarDynamicNoncooperativeGame1998}.

To analytically resolve the surrogate game, we first derive the optimal moments under the continuous-time delay-free spaces $\tilde{\mathcal{A}}_p \times \tilde{\mathcal{B}}_p$. These continuous-time optimal mappings will later be strictly projected or truncated into the piecewise-constant ZOH execution mechanism during practical implementation (Algorithm \ref{alg:g_lq_1}) and ZOH strategy replacement errors (Lemma \ref{lem:smooth_replacement}). 

Next, we introduce a symmetric matrix process $P(t)$ uniquely satisfying the following backward Generalized Riccati Differential Equation (GRDE):
\begin{equation}\label{eq:generalized_riccati}
    \left\{\begin{aligned}
        -\dot{P}(t) = & A^\top P(t) + P(t)A + Q - P(t)B_1 R_1^{-1} B_1^\top P(t) \\
        & + P(t)B_2 \left(R_2 + \lambda_2(t) I\right)^{-1} B_2^\top P(t) + \lambda_2(t)\gamma_2^2 I, \\
        P(T) = & Q_T.
    \end{aligned}\right.
\end{equation}
In this formulation, $\lambda_2(t)$ acts as a dynamic penalty factor tracking the local system divergence, defined strictly as:
\begin{equation}\label{eq:lambda_2}
    \lambda_2(t) \triangleq \max\left(0, \lambda_{\max}(Q_{2,eq}(t))\right).
\end{equation}
Here, the equivalent weight matrices $Q_{i,eq}(t)$ for $i \in \{1,2\}$ are elegantly constructed as:
\begin{equation}\label{eq:equivalent_Q}
    Q_{i,eq}(t) \triangleq \delta_{\pi}(t)B_i^\top P(t)B_i - (-1)^i R_i.
\end{equation}
These equivalent matrices bear a fundamental control-theoretic interpretation that will be mathematically justified in the subsequent optimal synthesis: they explicitly combine the baseline steady-state control costs ($R_i$) with the dynamic value-to-go penalty ($B_i^\top P(t)B_i$) induced by the stochastic diffusion propagating over the delay duration $\delta_\pi(t)$.

With $P(t)$ and the equivalent matrices formally defined, we restrict the game to a well-posed parameter regime where the linear-quadratic properties do not collapse due to extreme hardware limitations.

\begin{assumption}[Sufficient actuation capacity]\label{assump:actuation_capacity}
    We assume the physical energy capacities $\gamma_1$ and $\gamma_2$ are sufficiently large to support the unconstrained optimal feedback gains, satisfying the following matrix inequalities for all $t \in [t_0, T]$:
    \begin{equation}\label{eq:capacity_condition}
        \left\{\begin{aligned}
            \gamma_1^2 I - P(t)B_1 R_1^{-2}B_1^\top P(t) &\succeq \mathbf{0}, \\
            \gamma_2^2 I - P(t)B_2 (R_2 + \lambda_2(t) I)^{-2} B_2^\top P(t) &\succeq \mathbf{0}.
        \end{aligned}\right.
    \end{equation}
\end{assumption}
Condition \eqref{eq:capacity_condition} ensures that the actuators possess adequate power margins relative to the system's dynamic requirements. If violated, the optimal mean controls would inevitably suffer from strict boundary saturation (e.g., $\|\tilde{\Gamma}_i\|^2 = \gamma_i^2\|x\|^2$). Such persistent saturation destroys the global quadratic structure, degenerating the LQ game into a highly nonlinear bounded-control problem. Assumption \ref{assump:actuation_capacity} explicitly bounds the parameter regime where the analytical properties of the LQ paradigm are preserved.

\begin{theorem}[Game value and optimal surrogate strategies]\label{thm:dg-lq-solution}
    Let $v_2(t)$ be the unit principal eigenvector corresponding to $\lambda_{\max}(Q_{2,eq}(t))$. Under Assumption \ref{assump:actuation_capacity}, the game value and the optimal pure strategies of $\mathbf{\tilde{G}}_{lq}$ over $\tilde{\mathcal{A}}_p\times\tilde{\mathcal{B}}_p$ admit exact analytical solutions:
    \begin{enumerate}
        \item[(i)] The saddle-point game value is strictly a global quadratic form:
    \begin{equation*}
        \tilde{V}_p(t,x) = x^\top P(t)x, \quad \forall t\in[t_0,T].
    \end{equation*}
        \item[(ii)] For Player 1 (Minimizer), the optimal strategy is:
    \begin{equation}\label{eq:opt_strategy_1_new}
        \tilde{\Gamma}_1^*(t) = -R_1^{-1}B_1^\top P(t)\tilde{x}(t), \quad \tilde{\Sigma}_1^*(t) = \mathbf{0}.
    \end{equation}
        \item[(iii)] For Player 2 (Maximizer), let $K_2(t) \triangleq (R_2 + \lambda_2(t) I)^{-1} B_2^\top P(t)$. The optimal strategy is:
    \begin{equation}\label{eq:opt_strategy_2_new}
        \left\{\begin{aligned}
             \tilde{\Gamma}_2^*(t) &= K_2(t)\tilde{x}(t),\\
             \tilde{\Sigma}_2^*(t) &= \tilde{x}(t)^\top \left(\gamma_2^2 I - K_2(t)^\top K_2(t)\right) \tilde{x}(t) \cdot v_2(t)v_2(t)^\top \\
             &\quad \cdot \mathbb{I}_{\{\lambda_{\max}(Q_{2,eq}(t)) > 0\}}.
         \end{aligned}\right.
    \end{equation}
    \end{enumerate}
\end{theorem}
\begin{proof}
    We solve the game by analyzing the HJI equation of the surrogate SDG. By postulating a quadratic value function $\tilde{V}(t, x) = x^\top P(t) x$, the HJI equation decouples into two independent subproblems. We resolve Player 2's maximizer subproblem using Lagrangian duality, establishing that the optimal mean is a linear feedback and the variance must align with the principal eigenvector of the variance weight matrix $Q_{2,eq}$. Player 1's minimizer subproblem is solved via standard minimization. Substituting the optimal controls back into the HJI equation yields the GRDE \eqref{eq:generalized_riccati}, characterizing the evolution of $P(t)$. Full details are provided in Appendix \ref{app:thm:dg-lq-solution}.
\end{proof}
\begin{remark}[Violation of actuation capacity]\label{rmk:capacity_violation}
    Theorem \ref{thm:dg-lq-solution} relies strictly on Assumption \ref{assump:actuation_capacity} to guarantee that the optimal value function retains a globally explicit quadratic structure. In practical applications where the physical energy budget $\gamma_2$ is severely restricted and fails to satisfy condition \eqref{eq:capacity_condition}, the residual matrix $\gamma_2^2 I - K_2(t)^\top K_2(t)$ may lose its positive semi-definiteness, structurally invalidating the unconstrained Riccati derivation. Under this specific scenario, one must resort to numerically resolving the underlying HJI partial differential equation over state grids, as the true value function inevitably transitions from a global quadratic form into a saturated, highly nonlinear structure.
\end{remark}

Theorem \ref{thm:dg-lq-solution} reveals a profound \textit{dynamic energy allocation law}: Player 2 prioritizes their energy budget strictly for the deterministic linear feedback mean $\tilde{\Gamma}_2^*$. Only when the local dynamics approach divergence ($\lambda_{\max}(Q_{2,eq}) > 0$) does Player 2 inject the residual energy budget into the variance along the system's most vulnerable eigendirection $v_2(t)$. 
Furthermore, as the commitment delay $\bar\pi \to 0$, we naturally have $\lambda_2(t) \to 0$. Consequently, the optimal variance vanishes ($\tilde{\Sigma}_{2}^*(t) = \mathbf{0}$), and the GRDE \eqref{eq:generalized_riccati} degenerates to the standard continuous-time Riccati equation, fully recovering classical LQ theory.

\begin{algorithm}[ht]
\renewcommand{\algorithmicrequire}{\textbf{Input:}}
\renewcommand{\algorithmicensure}{\textbf{Output:}}
\caption{Near-Optimal Mixed Strategies for $\mathbf{G}_{lq}$}
\label{alg:g_lq_1}
\begin{algorithmic}[1]
\Require Dynamics $A, B_1, B_2$; costs $Q, R_1, R_2, Q_T$; delay $\pi$; capacities $\gamma_1, \gamma_2$; optional spatial grids $\mathcal{X}$.
\State \(\triangleright\) \textbf{Phase 1: Offline verification \& routing}
\State Solve GRDE \eqref{eq:generalized_riccati} for $P(t)$, $t \in [t_0, T]$.
\If{Assumption \ref{assump:actuation_capacity} holds via \eqref{eq:capacity_condition}}
    \State $\text{Mode} \leftarrow \text{Analytical (LQ)}$
\Else
    \State $\text{Mode} \leftarrow \text{Numerical (HJI Fallback)}$
    \State Solve HJI PDE over grids $\mathcal{X}$ for true value $V(t,x)$.
\EndIf
\State Initialize $x(t_0) = x_0$ and $\operatorname{cost}(t_0) \leftarrow 0$.
\For {$k=0, \ldots, N-1$}
    \State \(\triangleright\) \textbf{Phase 2: Online moment evaluation}
    \If{$\text{Mode} == \text{Analytical (LQ)}$}
        \State Compute $(\Gamma_{1,k}^*, \Sigma_{1,k}^*)$ and $(\Gamma_{2,k}^*, \Sigma_{2,k}^*)$ 
        \Statex \quad\quad via explicit solutions \eqref{eq:opt_strategy_1_new} and \eqref{eq:opt_strategy_2_new}.
    \Else
        \State $V_x \leftarrow \nabla_x V(t_k, x(t_k))$, $V_{xx} \leftarrow \nabla_{xx}^2 V(t_k, x(t_k))$.
        \State $M_{1,k} \leftarrow R_1 + \frac{1}{2}\delta_\pi(t_k) B_1^\top V_{xx} B_1$.
        \State $M_{2,k} \leftarrow \frac{1}{2}\delta_\pi(t_k) B_2^\top V_{xx} B_2 - R_2$.
        
        \State Solve Minimizer's optimization:
        \State $\min_{\Gamma_1, \Sigma_1 \succeq \mathbf{0}} \!\big\{ \Gamma_1^\top R_1 \Gamma_1 + V_x^\top B_1 \Gamma_1 + \tr(M_{1,k} \Sigma_1) \big\}$ 
        \State \quad s.t. $\|\Gamma_1\|^2 + \tr(\Sigma_1) \le \gamma_1^2 \|x(t_k)\|^2$.
        
        \State Solve Maximizer's optimization:
        \State $\max_{\Gamma_2, \Sigma_2 \succeq \mathbf{0}} \!\big\{ \!-\!\Gamma_2^\top R_2 \Gamma_2 + V_x^\top B_2 \Gamma_2 + \tr(M_{2,k} \Sigma_2) \big\}$
        \State \quad s.t. $\|\Gamma_2\|^2 + \tr(\Sigma_2) \le \gamma_2^2 \|x(t_k)\|^2$.
        
        \State Let $(\Gamma_{i,k}^*, \Sigma_{i,k}^*)$ be the resulting optimal moments.
    \EndIf
    \State \(\triangleright\) \textbf{Phase 3: ZOH inter-sample execution}
    \State Set prob. measures: $\mu_k \!=\! \mathcal{N}(\Gamma_{1,k}^*, \Sigma_{1,k}^*)$, $\nu_k \!=\! \mathcal{N}(\Gamma_{2,k}^*, \Sigma_{2,k}^*)$.
    \State Sample independently: $u_k \sim \mu_k$, $v_k \sim \nu_k$. 
    \For {$t \in [t_k, t_{k+1})$}
        \State Evolve dynamics: $\dot{x}(t) = Ax(t) + B_1 u_k + B_2 v_k$.
    \EndFor
    \State $J_{step} \leftarrow \int_{t_k}^{t_{k+1}} h(s, x(s), u_k, v_k) ds$.
    \State $\operatorname{cost}(t_{k+1}) \leftarrow \operatorname{cost}(t_k) + J_{step}$.
\EndFor
\State Add terminal cost: $\operatorname{cost}(T) \leftarrow \operatorname{cost}(T) + x(T)^\top Q_T x(T)$.
\Ensure Trajectories $x(t), u(t), v(t)$, and accumulated $\operatorname{cost}(T)$.
\end{algorithmic}
\end{algorithm}

\subsection{Synthesis of Near-Optimal Strategies}
Theorem \ref{thm:dg-lq-solution} provides explicit analytical solutions for the surrogate game $\mathbf{\tilde{G}}_{lq}$. A major computational advantage of this LQ framework is that the synthesis process relies entirely on solving the GRDE \eqref{eq:generalized_riccati} offline, fundamentally circumventing the curse of dimensionality inherent in general PDEs. 

To seamlessly integrate the analytical GRDE solutions with the strict physical capacity limits addressed in Remark \ref{rmk:capacity_violation}, we propose a robust \textit{dual-routing execution architecture}. As detailed in Algorithm \ref{alg:g_lq_1}, this architecture bridges continuous-time theory with discrete-time digital control practice through three distinct operational phases.

\textbf{Phase 1: Offline verification and routing.} Before online deployment, the system evaluates the mathematical well-posedness of the unconstrained LQ framework. By pre-computing the GRDE \eqref{eq:generalized_riccati} over $[t_0, T]$, we rigorously verify whether the physical actuation capacities $(\gamma_1, \gamma_2)$ are sufficient to support the unconstrained optimal feedback gains (Assumption \ref{assump:actuation_capacity}). If the condition holds, the system is routed to the highly efficient analytical mode. Conversely, if the capacities are severely restricted, the optimal moments will inevitably suffer from boundary saturations, destroying the global quadratic structure. In this scenario, the system triggers the numerical fallback mode, resorting to a grid-based HJI PDE solver to pre-compute the true non-quadratic value function over spatial grids $\mathcal{X}$.

\textbf{Phase 2: Online moment evaluation.} During real-time operation at each sampling instant $t_k$, the players compute their optimal statistical moments based on the observed state $x(t_k)$. In the analytical mode, this requires merely trivial $\mathcal{O}(1)$ matrix-vector multiplications via the explicit closed-form solutions \eqref{eq:opt_strategy_1_new} and \eqref{eq:opt_strategy_2_new}, rendering it extremely computationally efficient. In the numerical fallback mode, the players extract the local gradient and Hessian $(V_x, V_{xx})$ from the pre-computed HJI value function to construct the equivalent local weight matrices $M_{1,k}$ and $M_{2,k}$. The moments are then explicitly extracted by solving localized, state-dependent trace-optimization problems subject to the strict energy boundaries.

\textbf{Phase 3: ZOH inter-sample execution.} Having uniquely determined the optimal means and variances $(\Gamma_{i,k}^*, \Sigma_{i,k}^*)$, the players inversely project these deterministic moments back into the randomized mixed-strategy spaces. Specifically, they construct multivariate Gaussian distributions $\mathcal{N}(\Gamma_{i,k}^*, \Sigma_{i,k}^*)$ and independently draw their control actions $u_k$ and $v_k$. These sampled controls are then applied to the physical system under a Zero-Order Hold (ZOH) mechanism over the commitment interval $[t_k, t_{k+1})$, naturally inducing a macroscopic stochastic diffusion that physically replicates the continuous-time surrogate game dynamics.

\subsection{Performance Certification}
Having synthesized the near-optimal mixed strategies, we now rigorously certify the performance of the proposed methodology by evaluating the bounds on the value approximation error and the suboptimality gaps.

\begin{assumption}[Lipschitz mixed strategies]\label{assump:lipschitz_mixed}
    There exists a pair of saddle-point mixed strategies $(\alpha^*,\beta^*)$ for the original discrete-time game $\mathbf{G}_{lq}$ such that their projected statistical moments $(\phi_1(\alpha^*),\phi_2(\beta^*))$ are uniformly Lipschitz continuous.
\end{assumption}

\begin{remark}[Regularity of optimal mixed strategies]
    Unlike unconstrained pure strategies, the mathematical regularity of the first two moments for optimal mixed strategies has not been extensively explored in existing differential game literature. Furthermore, as revealed in Theorem \ref{thm:dg-lq-solution}, the analytical variance of the surrogate SDG inherently involves indicator functions, implying that the exact optimal mappings may exhibit non-Lipschitz jumps on certain boundaries. Therefore, Assumption \ref{assump:lipschitz_mixed} serves as a pragmatic working hypothesis. It effectively isolates the intrinsic theoretical weak approximation performance of our framework from the profound analytical complexities of non-smooth controls, rendering the suboptimality certification mathematically tractable.
\end{remark}

As established in Theorem \ref{thm:g_lq_weak_approximation}, the surrogate game $\mathbf{\tilde{G}}_{lq}$ under pure spaces $\tilde{\mathcal{A}}_p^\pi\times\tilde{\mathcal{B}}_p^\pi$ is an order $2$ weak approximation of the original game $\mathbf{G}_{lq}$ under mixed spaces $\mathcal{A}_m^\pi\times\mathcal{B}_m^\pi$. Let $(\tilde{\alpha}^*, \tilde{\beta}^*)$ be the optimal pure strategies for $\mathbf{\tilde{G}}_{lq}$ derived in Theorem \ref{thm:dg-lq-solution}. 

We define the continuous-time best responses against the ZOH-executed opponent strategies as $\tilde{\alpha}^{br} = \arg\min_{\tilde{\alpha}\in\tilde{\mathcal{A}}_p^{\pi}} \tilde{J}(t_0, x_0, \tilde{\alpha}, \operatorname{ZOH}_{\pi}[\tilde{\beta}^*])$ and $\tilde{\beta}^{br} = \arg\max_{\tilde{\beta}\in\tilde{\mathcal{B}}_p^{\pi}} \tilde{J}(t_0, x_0, \operatorname{ZOH}_{\pi}[\tilde{\alpha}^*], \tilde{\beta})$.

To uniquely isolate the theoretical weak approximation accuracy from the physical implementation error, we define the following ZOH execution penalties:
\begin{equation*}
    \begin{aligned}
        \epsilon_1 =& |\tilde{J}(t_0,x_0,\tilde{\alpha}^*,\phi_2(\beta^*))-\tilde{J}(t_0,x_0,\operatorname{ZOH}_{\pi}[\tilde{\alpha}^*],\phi_2(\beta^*))|,\\
        \epsilon_2 =& |\tilde{J}(t_0,x_0,\phi_1(\alpha^*),\tilde{\beta}^*)-\tilde{J}(t_0,x_0,\phi_1(\alpha^*),\operatorname{ZOH}_{\pi}[\tilde{\beta}^*])|,\\
        \epsilon_3 =& |\tilde{J}(t_0,x_0,\tilde{\alpha}^*,\tilde{\beta}^*)-\tilde{J}(t_0,x_0,\operatorname{ZOH}_{\pi}[\tilde{\alpha}^*],\operatorname{ZOH}_{\pi}[\tilde{\beta}^*])|,\\
        \epsilon_4 =& |\tilde{J}(t_0, x_0, \tilde{\alpha}^{br}, \operatorname{ZOH}_{\pi}[\tilde{\beta}^*]) - \tilde{J}(t_0, x_0, \tilde{\alpha}^{br}, \tilde{\beta}^*)|,\\
        \epsilon_5 =& |\tilde{J}(t_0, x_0, \operatorname{ZOH}_{\pi}[\tilde{\alpha}^*], \tilde{\beta}^{br}) - \tilde{J}(t_0, x_0, \tilde{\alpha}^*, \tilde{\beta}^{br})|.
    \end{aligned}
\end{equation*}

\begin{lemma}[ZOH strategy replacement errors]\label{lem:smooth_replacement}
    The cost differences caused by replacing the continuous-time strategies with their piecewise-constant ZOH equivalents are bounded by:
    \begin{equation}
        \epsilon_1 = \mathcal{O}(\bar{\pi}), \quad \epsilon_2 = \mathcal{O}(\bar{\pi}^{\frac{1}{2}}), \quad \epsilon_3 = \mathcal{O}(\bar{\pi}^{\frac{1}{2}}).
    \end{equation}
\end{lemma}
\begin{proof}
    We bound the replacement errors by analyzing the state tracking error $e(t) = \tilde{x}^d(t) - \tilde{x}^*(t)$ via Itô’s formula and Grönwall’s inequality. For $\epsilon_1$, the Lipschitz continuity of Player 1's optimal strategy ensures an $\mathcal{O}(\bar{\pi})$ convergence rate. For $\epsilon_2$, the indicator-function-dependent variance injection introduces non-Lipschitz jump discontinuities; we decouple this deviation into a quadratic state-mismatch error and a time-discretization error, showing that the non-smooth jumps degrade the tracking performance to $\mathcal{O}(\bar{\pi}^{1/2})$. The bound for $\epsilon_3$ follows directly from the dominance of this non-Lipschitz variance error in the combined strategy replacement.
    Full details are provided in Appendix \ref{app:lem:smooth_replacement}.
\end{proof}

\begin{lemma}[Best response deviation errors]\label{lem:best_response_deviation}
    The cost variations evaluating the continuous-time best responses against the continuous versus ZOH-executed opponent strategies are bounded by:
    \begin{equation}
        \epsilon_4 = \mathcal{O}(\bar{\pi}^{\frac{1}{2}}), \quad \epsilon_5 = \mathcal{O}(\bar{\pi}).
    \end{equation}
\end{lemma}
\begin{proof}
    We bound the best-response deviation by analyzing the state tracking error between the ideal continuous-time response and the ZOH-constrained implementation. For $\epsilon_4$, the opponent’s strategy inherits non-Lipschitz variance jumps (from Theorem \ref{thm:dg-lq-solution}), forcing a conservative Cauchy-Schwarz bound on the diffusion error, which yields $\mathcal{O}(\bar{\pi}^{1/2})$ convergence. Conversely, for $\epsilon_5$, Player 1's optimal strategy is zero-variance and globally Lipschitz, eliminating the diffusion discontinuities. This structure allows for a tighter $\mathcal{O}(\bar{\pi})$ bound derived via Grönwall’s inequality. Full details are provided in Appendix \ref{app:lem:best_response_deviation}.
\end{proof}

Building on these bounded execution penalties, we formally establish the global value approximation error and the strategy suboptimality gaps.

\begin{theorem}[Value approximation error]\label{thm:lq_value_approximation}
    Under Assumption \ref{assump:lipschitz_mixed}, the analytical game value $\tilde{V}_p(t_0,x_0)$ of $\mathbf{\tilde{G}}_{lq}$ approximates the true game value $V_m^\pi(t_0,x_0)$ of $\mathbf{G}_{lq}$ with an error bounded by:
    \begin{equation}
        |V_m^\pi(t_0,x_0)-\tilde{V}_p(t_0,x_0)| = \mathcal{O}(\bar\pi^{\frac{1}{2}}).
    \end{equation}
\end{theorem}
\begin{proof}
    According to the general SDG value certification framework established in our previous work \cite[Theorem 2]{xu2026optimalmixedstrategyzerosum}, the value approximation error is structurally bounded by the sum of the weak approximation error and the dominant execution penalty: $|V_m^\pi-\tilde{V}_p| \le \max\{\epsilon_1, \epsilon_2\} + \mathcal{O}(\bar{\pi}^n)$. 
    By Theorem \ref{thm:g_lq_weak_approximation}, the ZSLQDG weak approximation inherently guarantees an order $n=2$ accuracy, contributing an $\mathcal{O}(\bar{\pi}^2)$ error. Concurrently, Lemma \ref{lem:smooth_replacement} bounds the dominant execution penalty at $\epsilon_2 = \mathcal{O}(\bar{\pi}^{\frac{1}{2}})$. Summing these composite errors yields $\mathcal{O}(\bar{\pi}^{\frac{1}{2}}) + \mathcal{O}(\bar{\pi}^2) = \mathcal{O}(\bar{\pi}^{\frac{1}{2}})$. The proof is completed.
\end{proof}

\begin{theorem}[Strategy suboptimality gap]\label{thm:lq_strategy_suboptimality}
    The near-optimal mixed strategies $(\alpha^d, \beta^d)$ synthesized via Algorithm \ref{alg:g_lq_1} possess bounded suboptimality gaps against the worst-case responses:
    \begin{equation}
        \begin{aligned}
            J(t_0,x_0,\alpha^d,\beta^d) - \min_{\alpha\in\mathcal{A}_m^{\pi}}J(t_0, x_0, \alpha, \beta^d) &\leq \mathcal{O}(\bar{\pi}^{\frac{1}{2}}),\\
            \max_{\beta\in\mathcal{B}_m^{\pi}}J(t_0,x_0,\alpha^d,\beta)-J(t_0, x_0, \alpha^d, \beta^d) &\leq \mathcal{O}(\bar{\pi}^{\frac{1}{2}}).
        \end{aligned}
    \end{equation}
\end{theorem}
\begin{proof}
    Following the suboptimality certification framework derived in \cite[Theorem 3]{xu2026optimalmixedstrategyzerosum}, the performance gaps against the best responses are strictly upper bounded by the algebraic sum of the weak approximation error and the respective execution penalties. Specifically, the gap for Player 1 is bounded by $\epsilon_3 + \epsilon_4 + \mathcal{O}(\bar{\pi}^n)$, and for Player 2 by $\epsilon_3 + \epsilon_5 + \mathcal{O}(\bar{\pi}^n)$. 
    Substituting $n=2$ from Theorem \ref{thm:g_lq_weak_approximation} and incorporating the $\mathcal{O}(\bar{\pi}^{\frac{1}{2}})$ penalty bounds from Lemmas \ref{lem:smooth_replacement} and \ref{lem:best_response_deviation}, the lower-order execution penalty $\mathcal{O}(\bar{\pi}^{\frac{1}{2}})$ mathematically dominates the higher-order $\mathcal{O}(\bar{\pi}^2)$ theoretical gap in both dimensions. The suboptimality bounds are thus established.
\end{proof}

\section{Experiment}
In this section, we investigate a double-integrator pursuit-evasion problem in a 2-D plane. We cast the problem as a 4-dimensional ZSLQDG, which serves as a canonical pursuit-evasion benchmark in differential game literature. This model is chosen for its ability to represent the fundamental trade-off between acceleration-based control and state trajectory regulation, making it a classic representative of the broader class of linear-quadratic differential games.
After specifying the parameters for the system dynamics, cost functional, and admissible mixed strategy spaces, we numerically simulate the game under the proposed dual-routing near-optimal mixed strategies synthesized via Algorithm \ref{alg:g_lq_1}. 
Given varying commitment delay patterns, our empirical study serves a three-fold purpose: i) to investigate the physical influence of mixed strategies on the game states, control inputs, and accumulated costs; ii) to verify that the designed SDG under pure strategies strictly guarantees a weak approximation of the original game in the sense of both state distributions and expected costs; and iii) to evaluate the strategy suboptimality gap by pitting the proposed mixed strategies against continuous-time best responses (BR).

\subsection{2-D Double-Integrator Pursuit-Evasion Game}
Let the pursuer's state be $\left(x_p, v_p\right) \in \mathbb{R}^2 \times \mathbb{R}^2$ and the evader's state $\left(x_e, v_e\right)\in \mathbb{R}^2 \times \mathbb{R}^2$. Under planar double-integrator kinematics for time $t\in[t_0, T]$, the dynamics evolve as
\begin{equation}
    \left\{\begin{array}{rl}
         \dot{x}_i(t)=v_i(t),& x_i(t_0) = x_{i,0}\\ 
         \dot{v}_i(t)=a_i(t),& v_i(t_0) = v_{i,0}
    \end{array}\right.
\end{equation}
where $i\in\{p,e\}$ represents the pursuer or the evader, and $x_i, v_i, a_i\in \mathbb{R}^2$ denote position, velocity, and acceleration, respectively. By defining the relative state vector
\begin{equation}
    z(t)=\left[\begin{array}{l}
x_r(t) \\
v_r(t)
\end{array}\right]=\left[\begin{array}{l}
x_p(t)-x_e(t) \\
v_p(t)-v_e(t)
\end{array}\right] \in \mathbb{R}^4,
\end{equation}
we formulate the pursuit-evasion scenario as a $4$-dimensional ZSLQDG:
\begin{equation*}
    (\mathbf{G}_{pe})\!:\!\left\{\!\begin{aligned}
        &\dot{z}(t) = \underbrace{\left[\begin{array}{cc}
        \mathbf{0}_{2\times 2} & \mathbf{I}_2 \\
        \mathbf{0}_{2\times 2}  & \mathbf{0}_{2\times 2} 
        \end{array}\right]}_A z(t)+\underbrace{\left[\begin{array}{c}
        \mathbf{0}_{2\times 2}  \\
        \mathbf{I}_2
        \end{array}\right]}_{B_1} \underbrace{a_p(t)}_{u(t)}\\
        &\qquad\quad+\underbrace{\left[\begin{array}{c}
        \mathbf{0}_{2\times 2}  \\
        -\mathbf{I}_2
        \end{array}\right]}_{B_2} \underbrace{a_e(t)}_{v(t)}, \; z(t_0) = \underbrace{\begin{bmatrix}
            x_r(t_0)\\ 
            v_r(t_0)
        \end{bmatrix}}_{z_0},\\
        &J(t_0,z_0,u,v) = \mathbb{E}\int_{t_0}^T        \left[z^{\top}(t)Qz(t)+u^{\top}(t)R_1u(t)\right.\\
        &\qquad\qquad\qquad\left.- v^{\top}(t)R_2v(t)\right] dt + z^{\top}(T)Q_Tz(T),
    \end{aligned}\right.
\end{equation*}
where the expectation is taken over the underlying probability space induced by the mixed strategies. The matrix $Q\in\mathbb{R}^{4\times 4}$ penalizes the relative separation and closing-rate errors, $R_1, R_2\in\mathbb{R}^{2\times 2}$ penalize the pursuer's and evader's control efforts respectively, and $Q_T\in\mathbb{R}^{4\times 4}$ enforces the terminal-state penalty. The pursuer aims to minimize $J$ while the evader aims to maximize $J$, both adopting strategies from the admissible mixed strategy spaces in Definition \ref{def:admissible_mixed}. 

\subsection{Experiment Setup}
\paragraph*{Dynamics and Costs}
For the linear dynamics of $\mathbf{G}_{pe}$, we initialize the system at $z_0 = [1,1,0.6,0]^\top$, $t_0 = 0$, with a total horizon $T=10$. For the quadratic cost functional, the weighting matrices are configured as $Q=\mathbf{I}_4$, $R_1=0.1\mathbf{I}_2$, $R_2=0.1\mathbf{I}_2$, and $Q_T=\mathbf{I}_4$. Notably, assigning a relatively small control penalty $R_2$ deliberately intensifies the conflict, effectively incentivizing the evader to activate variance-injection behaviors (i.e., mixed strategies) to exploit the commitment delay.

\paragraph*{Admissible Mixed Strategies}
We investigate equispaced commitment delay partitions on $[t_0,T]$, such that the interval $\pi$ satisfies $\bar{\pi} = \underline{\pi}$, and $\bar\pi$ is evaluated across $\{0.06, 0.08, 0.1\}$. For the energy constraint, the physical capacities are conservatively bounded by $\gamma_1 = 30$ and $\gamma_2=5$. 

\paragraph*{Numerical Methods}
We employ the explicit Runge-Kutta method of order 5(4) to simulate the continuous dynamics of $\mathbf{G}_{pe}$, and the Euler-Maruyama method to discretize the stochastic dynamics of $\mathbf{\tilde{G}}_{pe}$, both operating at a fine time step of $\Delta t = 10^{-3}$. To accurately capture the macroscopic probability distributions of the state trajectories and the accumulated costs, $100$ independent Monte Carlo sample paths are simulated concurrently for each $\bar\pi$ configuration. 

Crucially, to verify the suboptimality gaps, we simulate the game under an asymmetric information structure. While the evaluated player is strictly restricted to the ZOH mixed strategy, the opponent is permitted to continuously observe the state and execute a continuous-time optimal feedback law (acting as the best response). This adversarial proxy mathematically guarantees strict upper and lower empirical bounds for the accumulated costs.

\subsection{Results and Analysis}

\subsubsection{Influence of Mixed Strategies}
\begin{figure}[h]
    \centering
    \subfigure[Relative position trajectory (Phase portrait).]{
        \includegraphics[width=0.46\linewidth]{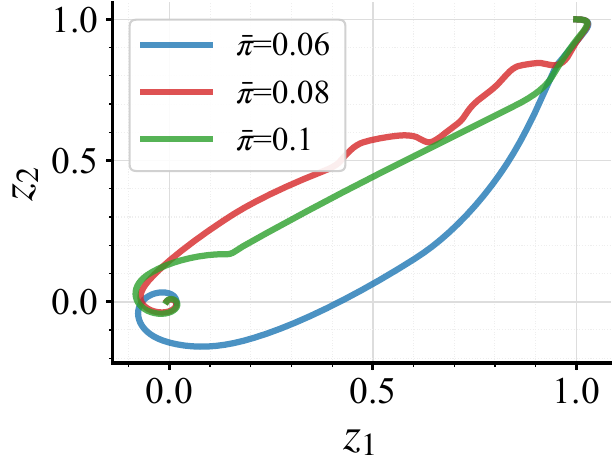}
        \label{fig:indi_state}
    }
    \subfigure[Accumulated cost (solid curve) and the approximated game value $\tilde{V}_p(t_0,z_0)$ (dotted line).]{
        \includegraphics[width=0.46\linewidth]{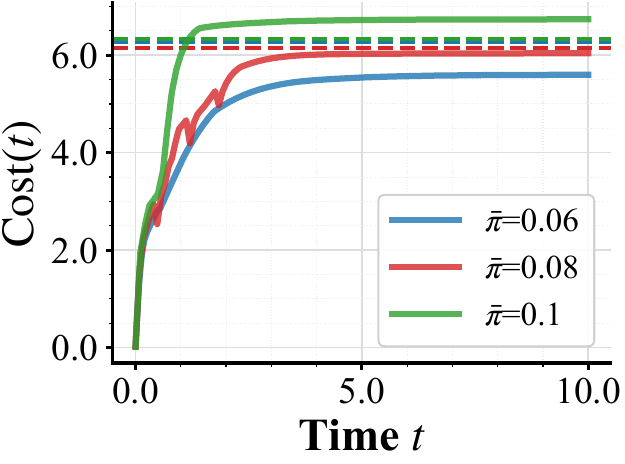}
        \label{fig:indi_cost}
    }
    \subfigure[Evader's control, dimension $1$.]{
        \includegraphics[width=0.46\linewidth]{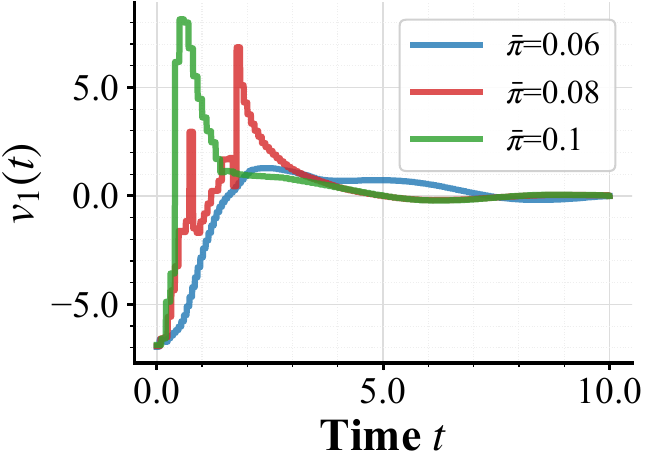}
        \label{fig:indi_control_1}
    }
    \subfigure[Evader's control, dimension $2$.]{
        \includegraphics[width=0.46\linewidth]{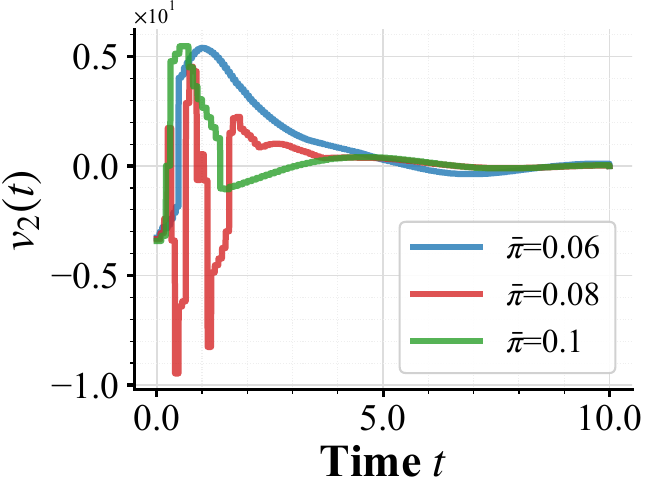}
        \label{fig:indi_control_2}
    }
    \subfigure[Pursuer's control, dimension $1$.]{
        \includegraphics[width=0.46\linewidth]{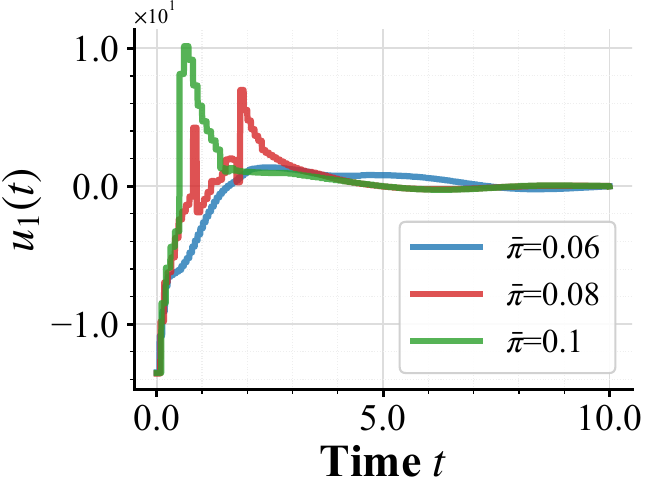}
        \label{fig:indi_control_3}
    }
    \subfigure[Pursuer's control, dimension $2$.]{
        \includegraphics[width=0.46\linewidth]{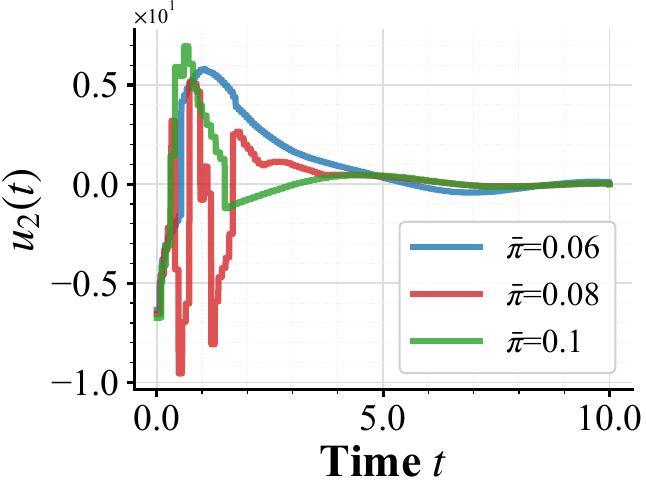}
        \label{fig:indi_control_4}
    }
    \caption{A realized sample trajectory of $\mathbf{G}_{pe}$, illustrating the physical manifestations of the near-optimal mixed strategies across different commitment frequencies.}
    \label{fig:individual_trajectory}
\end{figure}

Fig. \ref{fig:indi_state} presents the phase trajectory of the relative positions $\{x_r(t)\}_{t\in[t_0,T]}$ under near-optimal mixed strategies across varying commitment frequencies. Noticeably, as $\bar\pi$ increases (i.e., the commitment delay widens), the trajectory exhibits intensified stochastic exploration, diverging from the deterministic pure-strategy counterpart. 

Fig. \ref{fig:indi_cost} maps the real-time accumulated costs against the theoretical LQ game values (horizontal dotted lines). The relationship between the commitment delay $\bar\pi$ and the baseline game value reveals a profound structural property: the evader actively leverages mixed strategies to exploit the prolonged uncertainty induced by commitment delays, thereby elevating the guaranteed saddle-point value. Furthermore, Fig. \ref{fig:indi_control_1} and Fig. \ref{fig:indi_control_2} plot the realized piecewise-constant control signals of the evader, explicitly demonstrating the artificial variance injected by the normal distributions to disguise their intent. In contrast, as shown in Fig. \ref{fig:indi_control_3} and \ref{fig:indi_control_4}, although the pursuer's realized control actions also exhibit fluctuations, this apparent stochasticity is strictly propagated from the state trajectory. Because the pursuer's intrinsic injected variance is zero ($\tilde{\Sigma}_1^* = \mathbf{0}$), their control law remains a purely deterministic mapping of the stochastic state, perfectly corroborating the theoretical degeneration dictated by the dynamic energy allocation law.

\subsubsection{SDG Weak Approximation Verification}
\begin{figure}[h]
    \centering
    \subfigure[$1$-Wasserstein distance between the state distributions.]{
        \includegraphics[width=0.46\linewidth]{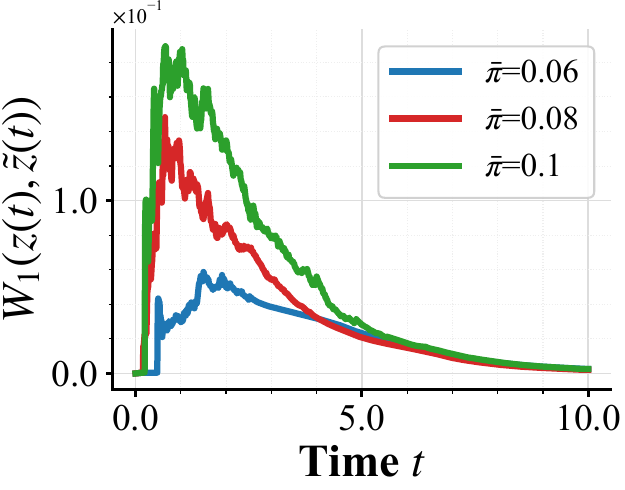}
        \label{fig:approx_state_w1}
    }
    \subfigure[Euclidean distance between the state expectations.]{
        \includegraphics[width=0.46\linewidth]{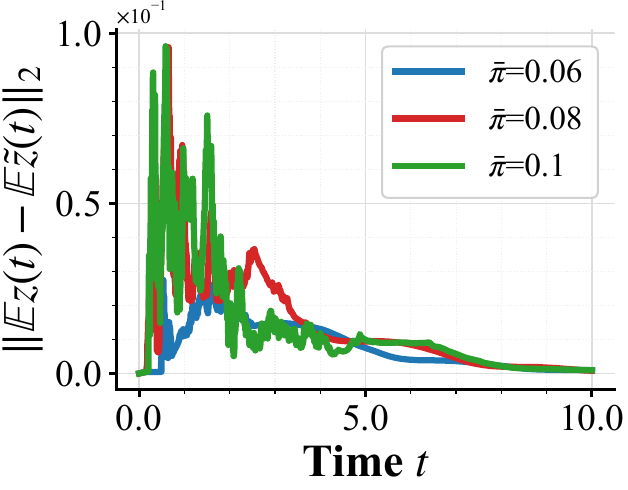}
        \label{fig:approx_state_mean}
    }
    \subfigure[Frobenius distance between the state covariances.]{
        \includegraphics[width=0.46\linewidth]{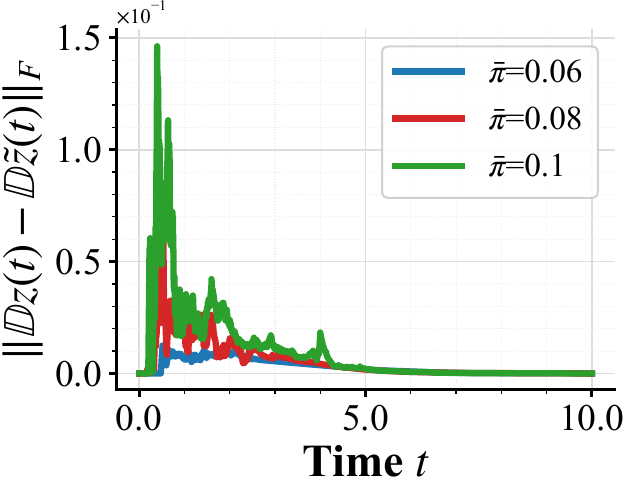}
        \label{fig:approx_state_cov}
    }
    \subfigure[Euclidean distance between the expected accumulated costs.]{
        \includegraphics[width=0.46\linewidth]{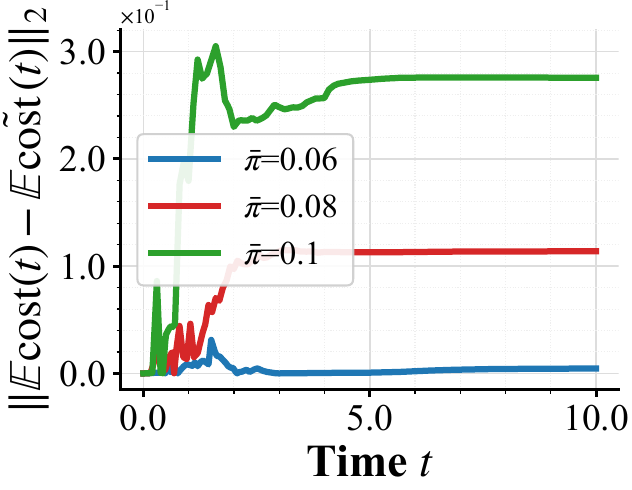}
        \label{fig:approx_cost_mean}
    }
    \caption{Weak approximation verification between the original game $\mathbf{G}_{pe}$ and the surrogate SDG $\mathbf{\tilde{G}}_{pe}$.}
    \label{fig:weak_approx}
\end{figure}

To empirically verify that the constructed surrogate SDG successfully achieves a weak approximation of the original mixed-strategy game, we analyze the distributional divergence between the true state $z(t)$ and the surrogate state $\tilde{z}(t)$. Fig. \ref{fig:approx_state_w1} computes the 1-Wasserstein metric $W_1(z(t), \tilde{z}(t)) = \inf_{\gamma} \mathbb{E}_{(z, \tilde{z})\sim\gamma} \| z - \tilde{z}\|_1$ over the horizon. The approximation errors remain strictly bounded and uniformly diminish as $\bar\pi$ decreases, settling at a scale of $10^{-1}$.

Delving into the moment evolutions, Fig. \ref{fig:approx_state_mean} and \ref{fig:approx_state_cov} isolate the approximation errors in the state expectations and covariances, respectively. The empirical results confirm that the first two statistical moments of the discrete ZOH mixed strategy are precisely tracked by the continuous SDG diffusion, with deviations strictly bounded at the $10^{-2}$ scale. This precise moment-matching practically corroborates the high-order weak approximation accuracy analytically established in Theorem \ref{thm:g_lq_weak_approximation}. Furthermore, Fig. \ref{fig:approx_cost_mean} verifies the required consistency in the expected accumulated costs, thereby fulfilling the performance criterion of the SDG weak approximation outlined in Definition \ref{def:SDG_weak_approximation}.

\subsubsection{Suboptimality Gap via Continuous Best Responses}
\begin{figure}[h]
    \centering
    \includegraphics[width=0.8\linewidth]{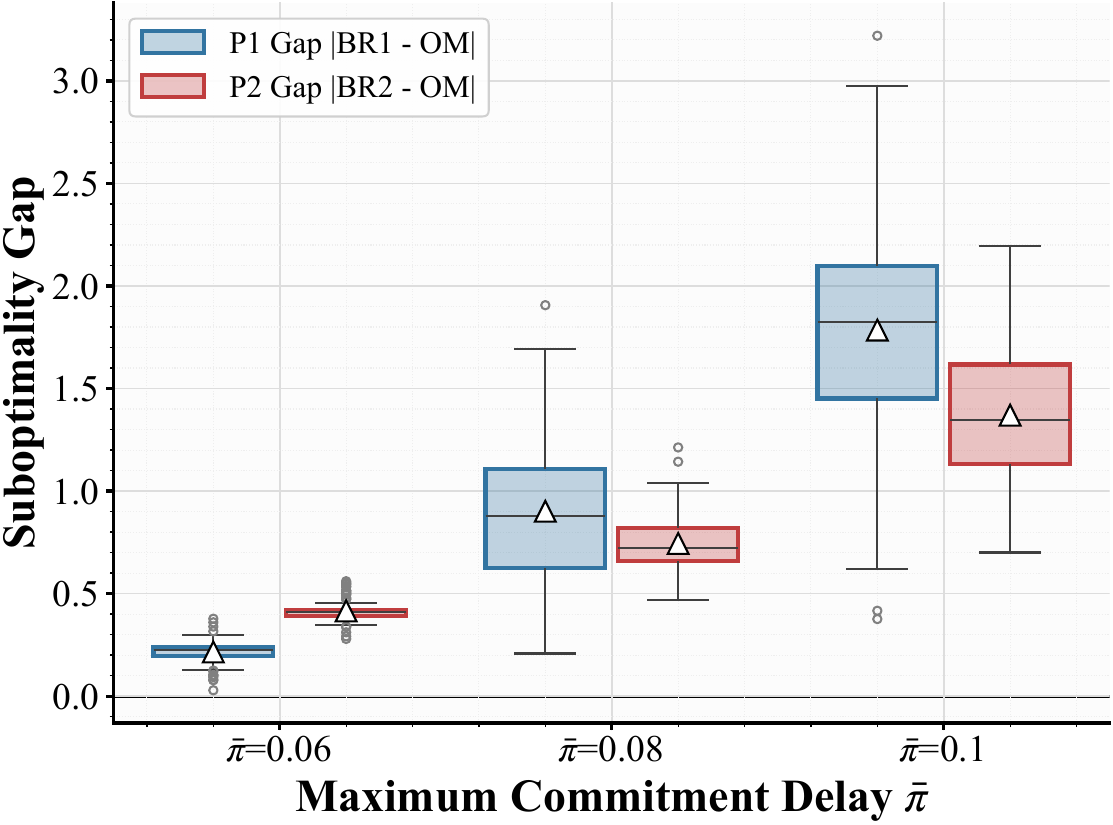}
    \caption{Empirical evaluation of the strategy suboptimality. By relaxing the ZOH constraint for one player to execute a continuous-time best response, the paired performance gap is robustly bounded.}
    \label{fig:suboptimality_gap}
\end{figure}

The suboptimality gap characterizes the maximum performance degradation a player might suffer when adopting the proposed near-optimal mixed strategy, assuming the opponent executes a worst-case optimal counter-strategy. To empirically quantify this gap and certify the theoretical bounds proposed in Theorem \ref{thm:lq_strategy_suboptimality}, we evaluate our synthesized strategies under an asymmetric information structure. Specifically, while the evaluated player is strictly restricted to the discrete-time ZOH mixed strategy, the adversarial opponent is granted continuous-time state observations to execute the ideal delay-free analytical feedback law. Because this unconstrained continuous-time feedback strictly outperforms any discrete-time counter-strategy, it serves as a conservative, cost-minimizing (or maximizing) proxy for the true best response.

To eliminate the cross-sample variance induced by the stochastic policies, we align the Monte Carlo random seeds, allowing us to evaluate the exact paired deviation over identical Brownian paths. Fig. \ref{fig:suboptimality_gap} maps the distributions of the empirical suboptimality gaps for Player 1 and Player 2. The boxplots demonstrate that despite facing a continuously observing and adapting adversary, the performance degradation incurred by the ZOH commitment delay strictly decreases. The gaps progressively vanish toward zero as the $\bar\pi$ scales down, corroborating the theoretically guaranteed $\mathcal{O}(\bar{\pi}^{\frac{1}{2}})$ convergence rate of the suboptimality gap established in Theorem \ref{thm:lq_strategy_suboptimality}.

\section{Conclusions}
In this paper, we have developed an analytical method for synthesizing near-optimal mixed strategies in zero-sum linear-quadratic differential games (ZSLQDGs). By constructing a surrogate pure-strategy stochastic differential game via second-order moment matching, we have effectively bypassed the intractability of optimizing over infinite-dimensional probability measure spaces. The resulting optimal strategy laws are governed by a Generalized Riccati Differential Equation (GRDE), which unveils a fundamental dynamic energy allocation law for variance injection in adversarial control.

We have demonstrated that this surrogate construction yields a high-order $\mathcal{O}(\bar{\pi}^2)$ weak approximation of the original system dynamics. Furthermore, we rigorously certified that both the global value approximation error and the strategy suboptimality gaps are bounded by $\mathcal{O}(\bar{\pi}^{1/2})$, providing explicit performance guarantees for our proposed dual-routing execution architecture. Our results offer robust analytical benchmarks for ZSLQDGs and clarify the physical mechanisms of randomization in continuous-time conflict. Extending this moment-matching framework to $N$-player general-sum differential games is a promising direction for future research, where the equilibrium analysis must generalize from a unique saddle point to more complex Nash equilibria.

\bibliographystyle{IEEEtran}
\bibliography{refs.bib}

@book{anderson2007optimal,
  title={Optimal control: linear quadratic methods},
  author={Anderson, Brian DO and Moore, John B},
  year={2007},
  publisher={Courier Corporation}
}

@article{braslavsky2007feedback,
  title={Feedback stabilization over signal-to-noise ratio constrained channels},
  author={Braslavsky, Julio H and Middleton, Richard H and Freudenberg, James S},
  journal={IEEE Transactions on automatic control},
  volume={52},
  number={8},
  pages={1391--1403},
  year={2007},
  publisher={IEEE}
}

@article{wonham1967optimal,
  title={Optimal stationary control of a linear system with state-dependent noise},
  author={Wonham, W Murray},
  journal={SIAM Journal on Control},
  volume={5},
  number={3},
  pages={486--500},
  year={1967},
  publisher={SIAM}
}

@article{basarUniquenessNashSolution1976,
  title = {On the Uniqueness of the Nash Solution in Linear-Quadratic Differential Games},
  author = {Basar, T.},
  year = {1976},
  journal = {International Journal of Game Theory},
  volume = {5},
  pages = {65--90},
  publisher = {Springer},
  urldate = {2025-05-28}
}

@article{jacobsonOptimalStochasticLinear2003,
  title = {Optimal Stochastic Linear Systems with Exponential Performance Criteria and Their Relation to Deterministic Differential Games},
  author = {Jacobson, David},
  year = {2003},
  journal = {IEEE Transactions on Automatic control},
  volume = {18},
  number = {2},
  pages = {124--131},
  publisher = {IEEE},
  urldate = {2025-05-28}
}

@article{weerenAsymptoticAnalysisLinear1999,
  title = {Asymptotic Analysis of Linear Feedback Nash Equilibria in Nonzero-Sum Linear-Quadratic Differential Games},
  author = {Weeren, ARIE JTM and Schumacher, Johannes M. and Engwerda, Jacob C.},
  year = {1999},
  journal = {Journal of Optimization Theory and Applications},
  volume = {101},
  pages = {693--722},
  publisher = {Springer},
  urldate = {2025-05-28}
}

@article{aggarwalLinearQuadraticZeroSum2024,
  title = {Linear Quadratic Zero-Sum Differential Games With Intermittent and Costly Sensing},
  author = {Aggarwal, Shubham and Ba{\c s}ar, Tamer and Maity, Dipankar},
  year = {2024},
  journal = {IEEE Control Systems Letters},
  volume = {8},
  pages = {1601--1606},
  issn = {2475-1456},
  urldate = {2025-05-28}
}

@article{hoDifferentialGamesOptimal1965a,
  title = {Differential Games and Optimal Pursuit-Evasion Strategies},
  author = {Ho, Y. and Bryson, A. and Baron, S.},
  year = {1965},
  month = oct,
  journal = {IEEE Transactions on Automatic Control},
  volume = {10},
  number = {4},
  pages = {385--389},
  issn = {1558-2523}
}

@article{jagatNonlinearControlSpacecraft2017,
  title = {Nonlinear Control for Spacecraft Pursuit-Evasion Game Using the State-Dependent Riccati Equation Method},
  author = {Jagat, Ashish and Sinclair, Andrew J.},
  year = {2017},
  month = dec,
  journal = {IEEE Transactions on Aerospace and Electronic Systems},
  volume = {53},
  number = {6},
  pages = {3032--3042},
  issn = {1557-9603},
  urldate = {2025-05-23}
}

@article{liDefendingAssetLinear2011,
  title = {Defending an Asset: A Linear Quadratic Game Approach},
  shorttitle = {Defending an Asset},
  author = {Li, Dongxu and Cruz, Jose B.},
  year = {2011},
  month = apr,
  journal = {IEEE Transactions on Aerospace and Electronic Systems},
  volume = {47},
  number = {2},
  pages = {1026--1044},
  issn = {1557-9603},
  urldate = {2025-05-25}
}

@inproceedings{menonGuidanceLawsSpacecraft1988,
  title = {Guidance Laws for Spacecraft Pursuit-Evasion and Rendezvous},
  booktitle = {Guidance, Navigation and Control Conference},
  author = {Menon, P. and Calise, A.},
  year = {1988},
  month = aug,
  publisher = {{American Institute of Aeronautics and Astronautics}},
  address = {Minneapolis,MN,U.S.A.},
  urldate = {2025-05-25},
  langid = {english}
}

@article{turetskyMissileGuidanceLaws2003,
  title = {Missile Guidance Laws Based on Pursuit--Evasion Game Formulations},
  author = {Turetsky, Vladimir and Shinar, Josef},
  year = {2003},
  month = apr,
  journal = {Automatica},
  volume = {39},
  number = {4},
  pages = {607--618},
  issn = {0005-1098},
  urldate = {2025-05-27}
}

@article{xu2026optimalmixedstrategyzerosum,
      title={Near-Optimal Mixed Strategy for Zero-Sum Differential Games}, 
      author={Tao Xu and Wang Xi and Jianping He},
      year={2026},
      journal={arXiv preprint arXiv:2308.01144} 
}

@incollection{aumann28MixedBehavior1964,
  title = {28. Mixed and Behavior Strategies in Infinite Extensive Games},
  booktitle = {Advances in Game Theory. (AM-52)},
  author = {Aumann, Robert J.},
  year = {1964},
  month = dec,
  pages = {627--650},
  publisher = {Princeton University Press},
  urldate = {2024-04-09},
  isbn = {978-1-4008-8201-4}
}

@article{aumannBorelStructuresFunction1961,
  title = {Borel Structures for Function Spaces},
  author = {Aumann, Robert J.},
  year = {1961},
  journal = {Illinois Journal of Mathematics},
  volume = {5},
  number = {4},
  pages = {614--630},
  publisher = {Duke University Press},
  urldate = {2024-04-09}
}

@article{aumannSpacesMeasurableTransformations1960,
  title = {Spaces of Measurable Transformations},
  author = {Aumann, Robert J.},
  year = {1960},
  journal = {Bulletin of the American Mathematical Society},
  volume = {66},
  number = {4},
  pages = {301--304},
  urldate = {2024-04-09}
}

@book{basarDynamicNoncooperativeGame1998,
  title = {Dynamic Noncooperative Game Theory},
  author = {Ba{\c s}ar, Tamer and Olsder, Geert Jan},
  year = {1998},
  publisher = {SIAM}
}

@incollection{berkovitzRelaxedControls2012,
  title = {Relaxed Controls},
  booktitle = {Nonlinear Optimal Control Theory},
  author = {Berkovitz, Leonard David and Medhin, Negash G.},
  year = {2012},
  publisher = {{Chapman and Hall/CRC}},
  isbn = {978-0-429-09810-9}
}

@article{buckdahnValueFunctionDifferential2013,
  title = {Value Function of Differential Games without Isaacs Conditions. An Approach with Nonanticipative Mixed Strategies},
  author = {Buckdahn, Rainer and Li, Juan and Quincampoix, Marc},
  year = {2013},
  month = nov,
  journal = {International Journal of Game Theory},
  volume = {42},
  number = {4},
  pages = {989--1020},
  issn = {1432-1270},
  urldate = {2023-03-04},
  langid = {english}
}

@article{buckdahnValueMixedStrategies2014,
  title = {Value in Mixed Strategies for Zero-Sum Stochastic Differential Games Without Isaacs Condition},
  author = {Buckdahn, Rainer and Li, Juan and Quincampoix, Marc},
  year = {2014},
  journal = {The Annals of Probability},
  volume = {42},
  number = {4},
  eprint = {42920505},
  eprinttype = {jstor},
  pages = {1724--1768},
  publisher = {Institute of Mathematical Statistics},
  issn = {0091-1798},
  urldate = {2023-03-29}
}

@article{cardaliaguetDifferentialGamesAsymmetric2007,
  title = {Differential Games with Asymmetric Information},
  author = {Cardaliaguet, P.},
  year = {2007},
  month = jan,
  journal = {SIAM Journal on Control and Optimization},
  volume = {46},
  number = {3},
  pages = {816--838},
  publisher = {{Society for Industrial and Applied Mathematics}},
  issn = {0363-0129},
  urldate = {2023-02-09}
}

@article{cardaliaguetPureRandomStrategies2014,
  title = {Pure and Random Strategies in Differential Game with Incomplete Informations},
  author = {Cardaliaguet, Pierre and Jimenez, Chlo{\'e} and Quincampoix, Marc},
  year = {2014},
  journal = {Journal of Dynamics and Games},
  volume = {1},
  number = {3},
  pages = {363--375},
  publisher = {{Journal of Dynamics and Games}},
  issn = {2164-6066},
  urldate = {2023-02-09},
  copyright = {http://creativecommons.org/licenses/by/3.0/},
  langid = {english}
}

@article{coxControlledMeasurevaluedMartingales2024,
  title = {Controlled Measure-Valued Martingales: A Viscosity Solution Approach},
  shorttitle = {Controlled Measure-Valued Martingales},
  author = {Cox, Alexander M. G. and K{\"a}llblad, Sigrid and Larsson, Martin and {Svaluto-Ferro}, Sara},
  year = {2024},
  month = apr,
  journal = {The Annals of Applied Probability},
  volume = {34},
  number = {2},
  pages = {1987--2035},
  publisher = {Institute of Mathematical Statistics},
  issn = {1050-5164, 2168-8737},
  urldate = {2024-05-04}
}

@article{elliottExistenceValueStochastic1976,
  title = {The Existence of Value in Stochastic Differential Games},
  author = {Elliott, Robert},
  year = {1976},
  month = jan,
  journal = {SIAM Journal on Control and Optimization},
  volume = {14},
  number = {1},
  pages = {85--94},
  publisher = {{Society for Industrial and Applied Mathematics}},
  issn = {0363-0129},
  urldate = {2023-03-28}
}

@book{engwerdaLQDynamicOptimization2005,
  title = {LQ Dynamic Optimization and Differential Games},
  author = {Engwerda, Jacob},
  year = {2005},
  month = jun,
  edition = {1st edition},
  publisher = {Wiley},
  address = {Chicester, West Sussex, England Hoboken, NJ},
  isbn = {978-0-470-01524-7},
  langid = {english}
}

@book{flemingControlledMarkovProcesses2006,
  title = {Controlled Markov Processes and Viscosity Solutions},
  author = {Fleming, Wendell H. and Soner, Halil Mete},
  year = {2006},
  month = feb,
  publisher = {Springer Science \& Business Media},
  googlebooks = {4Bjz2iWmLyQC},
  isbn = {978-0-387-31071-8},
  langid = {english}
}

@article{flemingMixedStrategiesDeterministic2017a,
  title = {Mixed Strategies for Deterministic Differential Games},
  author = {Fleming, Wendell H. and {Hernandez-Hernandez}, Daniel},
  year = {2017},
  month = jun,
  journal = {Communications on Stochastic Analysis},
  volume = {11},
  number = {2},
  issn = {0973-9599},
  urldate = {2023-05-17},
  langid = {english}
}

@inproceedings{goodfellowGenerativeAdversarialNets2014,
  title = {Generative Adversarial Nets},
  booktitle = {Advances in Neural Information Processing Systems},
  author = {Goodfellow, Ian and {Pouget-Abadie}, Jean and Mirza, Mehdi and Xu, Bing and {Warde-Farley}, David and Ozair, Sherjil and Courville, Aaron and Bengio, Yoshua},
  year = {2014},
  volume = {27},
  urldate = {2024-05-17}
}

@article{harrisExistenceSubgameperfectEquilibrium1995,
  title = {The Existence of Subgame-Perfect Equilibrium in Continuous Games with Almost Perfect Information: A Case for Public Randomization},
  shorttitle = {The Existence of Subgame-Perfect Equilibrium in Continuous Games with Almost Perfect Information},
  author = {Harris, Christopher and Reny, Philip and Robson, Arthur},
  year = {1995},
  journal = {Econometrica: Journal of the Econometric Society},
  eprint = {2171906},
  eprinttype = {jstor},
  pages = {507--544},
  publisher = {JSTOR},
  urldate = {2025-03-25}
}

@inproceedings{hespanhaProbabilisticPursuitevasionGames2000,
  title = {Probabilistic Pursuit-Evasion Games: A One-Step Nash Approach},
  shorttitle = {Probabilistic Pursuit-Evasion Games},
  booktitle = {Proceedings of the 39th IEEE Conference on Decision and Control},
  author = {Hespanha, J.P. and Prandini, M. and Sastry, S.},
  year = {2000},
  month = dec,
  volume = {3},
  pages = {2272-2277 vol.3},
  issn = {0191-2216}
}

@book{isaacsDifferentialGamesMathematical1999,
  title = {Differential Games: A Mathematical Theory with Applications to Warfare and Pursuit, Control and Optimization},
  shorttitle = {Differential Games},
  author = {Isaacs, Rufus},
  year = {1999},
  publisher = {Courier Corporation}
}

@article{islerRandomizedPursuitEvasionLocal2006,
  title = {Randomized Pursuit-Evasion with Local Visibility},
  author = {Isler, Volkan and Kannan, Sampath and Khanna, Sanjeev},
  year = {2006},
  month = jan,
  journal = {SIAM Journal on Discrete Mathematics},
  volume = {20},
  number = {1},
  pages = {26--41},
  publisher = {{Society for Industrial and Applied Mathematics}},
  issn = {0895-4801},
  urldate = {2023-06-19}
}

@article{islerRandomizedPursuitevasionPolygonal2005,
  title = {Randomized Pursuit-Evasion in a Polygonal Environment},
  author = {Isler, V. and Kannan, S. and Khanna, S.},
  year = {2005},
  month = oct,
  journal = {IEEE Transactions on Robotics},
  volume = {21},
  number = {5},
  pages = {875--884},
  issn = {1941-0468}
}

@article{kumarOptimalMixedStrategies1980,
  title = {Optimal Mixed Strategies in a Dynamic Game},
  author = {Kumar, P.},
  year = {1980},
  month = aug,
  journal = {IEEE Transactions on Automatic Control},
  volume = {25},
  number = {4},
  pages = {743--749},
  issn = {1558-2523},
  urldate = {2024-04-23}
}

@article{kunischOptimalControlUndamped2016,
  title = {Optimal Control of the Undamped Linear Wave Equation with Measure Valued Controls},
  author = {Kunisch, Karl and Trautmann, Philip and Vexler, Boris},
  year = {2016},
  month = jan,
  journal = {SIAM Journal on Control and Optimization},
  volume = {54},
  number = {3},
  pages = {1212--1244},
  publisher = {{Society for Industrial and Applied Mathematics}},
  issn = {0363-0129},
  urldate = {2024-05-04}
}

@article{levantChatteringAnalysis2010,
  title = {Chattering Analysis},
  author = {Levant, Arie},
  year = {2010},
  month = jun,
  journal = {IEEE Transactions on Automatic Control},
  volume = {55},
  number = {6},
  pages = {1380--1389},
  issn = {1558-2523},
  urldate = {2025-03-19}
}

@inproceedings{martinFindingMixedstrategyEquilibria2023,
  title = {Finding Mixed-Strategy Equilibria of Continuous-Action Games without Gradients Using Randomized Policy Networks},
  booktitle = {Proceedings of the Thirty-Second International Joint Conference on Artificial Intelligence},
  author = {Martin, Carlos and Sandholm, Tuomas},
  year = {2023},
  month = aug,
  series = {IJCAI '23},
  pages = {2844--2852},
  urldate = {2024-05-10},
  isbn = {978-1-956792-03-4}
}

@article{milshteinWeakApproximationSolutions1986a,
  title = {Weak Approximation of Solutions of Systems of Stochastic Differential Equations},
  author = {Mil'shtein, G. N.},
  year = {1986},
  journal = {Theory of Probability \& Its Applications},
  volume = {30},
  number = {4},
  pages = {750--766},
  publisher = {SIAM}
}

@article{perkinsMixedStrategyLearningContinuous2017a,
  title = {Mixed-Strategy Learning With Continuous Action Sets},
  author = {Perkins, Steven and Mertikopoulos, Panayotis and Leslie, David S.},
  year = {2017},
  month = jan,
  journal = {IEEE Transactions on Automatic Control},
  volume = {62},
  number = {1},
  pages = {379--384},
  issn = {1558-2523},
  urldate = {2024-05-06}
}

@inproceedings{petersLearningMixedStrategies2022,
  title = {Learning Mixed Strategies in Trajectory Games},
  booktitle = {Robotics: Science and Systems XVIII},
  author = {Peters, Lasse and {Fridovich-Keil}, David and Ferranti, Laura and Stachniss, Cyrill and {Alonso-Mora}, Javier and Laine, Forrest},
  year = {2022},
  month = jun,
  publisher = {{Robotics: Science and Systems Foundation}},
  urldate = {2024-04-23},
  isbn = {978-0-9923747-8-5},
  langid = {english}
}

@article{sirbuStochasticPerronsMethod2014,
  title = {Stochastic Perron's Method and Elementary Strategies for Zero-Sum Differential Games},
  author = {S{\^i}rbu, Mihai},
  year = {2014},
  month = jan,
  journal = {SIAM Journal on Control and Optimization},
  volume = {52},
  number = {3},
  pages = {1693--1711},
  publisher = {{Society for Industrial and Applied Mathematics}},
  issn = {0363-0129},
  urldate = {2024-04-07}
}

@article{vidalProbabilisticPursuitevasionGames2002a,
  title = {Probabilistic Pursuit-Evasion Games: Theory, Implementation, and Experimental Evaluation},
  shorttitle = {Probabilistic Pursuit-Evasion Games},
  author = {Vidal, R. and Shakernia, O. and Kim, H.J. and Shim, D.H. and Sastry, S.},
  year = {2002},
  month = oct,
  journal = {IEEE Transactions on Robotics and Automation},
  volume = {18},
  number = {5},
  pages = {662--669},
  issn = {2374-958X}
}

@article{wangPontryaginsMaximumPrinciple2010,
  title = {A Pontryagin's Maximum Principle for Non-Zero Sum Differential Games of BSDEs with Applications},
  author = {Wang, Guangchen and Yu, Zhiyong},
  year = {2010},
  month = jul,
  journal = {IEEE Transactions on Automatic Control},
  volume = {55},
  number = {7},
  pages = {1742--1747},
  issn = {1558-2523}
}

@inproceedings{xuDifferentialGameMixed2023b,
  title = {Differential Game with Mixed Strategies: A Weak Approximation Approach},
  shorttitle = {Differential Game with Mixed Strategies},
  booktitle = {2023 62nd IEEE Conference on Decision and Control (CDC)},
  author = {Xu, Tao and Xi, Wang and He, Jianping},
  year = {2023},
  month = dec,
  pages = {5216--5221},
  issn = {2576-2370},
  urldate = {2025-03-25}
}

@article{dou2019finding,
  title={Finding mixed strategy nash equilibrium for continuous games through deep learning},
  author={Dou, Zehao and Yan, Xiang and Wang, Dongge and Deng, Xiaotie},
  journal={arXiv preprint arXiv:1910.12075},
  year={2019}
}

\begin{appendices}

\section{Proof of Theorem \ref{thm:SE}} \label{app:thm:SE}
\begin{proof}
    The proof proceeds by backward induction. At the terminal time $t_N = T$, the game value uniquely exists as $V_m^\pi(T,x) = x^\top Q_T x$.

    Assume the unique game value $V_m^\pi(t_k, \cdot)$ exists at step $k$. For the local stage game over $[t_{k-1}, t_k)$ given $x(t_{k-1}) = x$, let $\mathcal{P}_x^{(1)}$ and $\mathcal{P}_x^{(2)}$ denote the convex sets of probability measures satisfying the energy constraints in Definition \ref{def:admissible_mixed} with capacities $\gamma_1$ and $\gamma_2$, respectively. Under Assumption \ref{ass:independence}, the local expected cost-to-go evaluates to:
    \begin{equation*}
        \begin{aligned}
            &J_{k-1}(x, \mu, \nu) \triangleq \int_{U \times V} \bigg\{ \int_{t_{k-1}}^{t_k} h(t, x(t), u, v) dt \\
            &\qquad\qquad\qquad\qquad + V_m^\pi(t_k, x(t_k)) \bigg\} \mu(du)\nu(dv),
        \end{aligned}
    \end{equation*}

    To equate the upper and lower values (i.e., $\inf_{\mu} \sup_{\nu} J_{k-1} = \sup_{\nu} \inf_{\mu} J_{k-1}$), we verify the requisite conditions of Sion's Minimax Theorem:
    1) \textit{Weak Compactness:} The uniform second-moment bounds in Definition \ref{def:admissible_mixed} guarantee that the measure families $\mathcal{P}_x^{(1)}$ and $\mathcal{P}_x^{(2)}$ are tight. By Prokhorov's Theorem, this tightness ensures relative weak compactness. Furthermore, since the second-moment functional is lower semi-continuous in the weak topology, $\mathcal{P}_x^{(i)}$ are closed, and thus strictly weakly compact.
    2) \textit{Bilinearity and Continuity:} As an expectation, the functional $J_{k-1}$ is strictly bilinear (hence concave-convex) and continuous with respect to $(\mu, \nu)$ in the weak topology.

    With all conditions satisfied, Sion's Minimax Theorem guarantees the existence of a unique local saddle point, confirming $V_{m,+}^\pi(t_{k-1}, x) = V_{m,-}^\pi(t_{k-1}, x) \triangleq V_m^\pi(t_{k-1}, x)$. By backward induction down to $t_0$, the global upper and lower value functions coincide, i.e., $V_{m,+}^\pi(t_0,x_0) = V_{m,-}^\pi(t_0,x_0) = V^\pi_m(t_0,x_0)$, completing the proof.
\end{proof}

\section{Proof of Theorem \ref{thm:g_lq_weak_approximation}}\label{app:thm:g_lq_weak_approximation}
\begin{proof}
Our certification of the SDG weak approximation proceeds via a two-stage procedure derived from the classical weak approximation theory of SDEs \cite{milshteinWeakApproximationSolutions1986a}: i) investigate the local one-step approximation errors between the original game dynamics and the designed SDG, and ii) if the local one-step approximation errors for the first three moments are bounded by $\mathcal{O}(\bar{\pi}^3)$, then the global approximation error over the entire time horizon possesses an order of $2$, i.e., $\mathcal{O}(\bar{\pi}^2)$.

Unlike classical SDE weak approximation, an SDG approximation requires both the state trajectory and the accumulated cost function to be jointly approximated. To facilitate this, we first transform the original game $\mathbf{G}_{lq}$ and the surrogate SDG $\mathbf{\tilde{G}}_{lq}$ into auxiliary forms where the running costs are absorbed into an augmented state variable $y$. Let $z(t) = [u(t)^\top, v(t)^\top]^\top$, $R \triangleq \diag(R_1, -R_2)$ and $B \triangleq [B_1, B_2]$. The auxiliary games are constructed as:
\begin{equation*}
    \begin{aligned}
        (\mathbf{G}_{lq}^{aux})\!:&\!\left\{\!\!\!\begin{array}{l}
             \dot{x}(t)=  Ax(t) + B z(t), \\
             \dot{y}(t) = \mathbb{E}\left[x^{\top}(t)Qx(t) + z^{\top}(t)Rz(t)\right],\\
             x(t_0)= x_0, \quad y(t_0) = 0,\\
             J = \mathbb{E}\left[y(T)+x^{\top}(T)Q_Tx(T)\right].
        \end{array}\right.\\
         (\mathbf{\tilde{G}}_{lq}^{aux})\!:&\!\left\{\!\!\!\begin{array}{l}
              d \tilde{x}(t) \!=\! (A\tilde{x}(t) + B\tilde{\Gamma}(t)) dt + \delta_{\pi}^{\frac{1}{2}}(t)B\tilde{\Lambda}(t)d W_t, \\
              d \tilde{y}(t) \!=\! \mathbb{E}\!\left[\tilde{x}^{\top}\!(t)Q\tilde{x}(t) \!+\! \tilde{\Gamma}^{\top}\!(t)R\tilde{\Gamma}(t) \!+\! \tr(\tilde{\Sigma}(t)R) \right]\!dt,\\
              \tilde{x}(t_0) \!=\! x_0, \quad \tilde{y}(t_0) \!=\! 0,\\
              \tilde{J} \!=\! \mathbb{E}\left[\tilde{y}(T)+\tilde{x}^{\top}(T)Q_T\tilde{x}(T)\right].
         \end{array}\right.
    \end{aligned}
\end{equation*}

We now proceed to the local one-step error analysis. Let $\delta = t_1 - t_0$ be the length of the first commitment interval, satisfying $\delta \leq \bar{\pi}$. Let $\Delta=[\Delta_x^{\top},\Delta_y]^{\top}$ be the exact one-step increment for $\mathbf{G}_{lq}^{aux}$, and $\tilde{\Delta}=[\tilde{\Delta}_x^{\top},\tilde{\Delta}_y]^{\top}$ be the increment for the surrogate $\mathbf{\tilde{G}}_{lq}^{aux}$.

\begin{lemma}[One-step local approximation]\label{lem:one_step_dg}
    Under the strategy projection map $\phi$ defined in Lemma \ref{def:map_lq}, the differences between the local moments of $\Delta$ and $\tilde{\Delta}$ conditional on the initial state are strictly bounded by:
    \begin{equation*}
        \begin{array}{l}
           \text{(i) } \| \mathbb{E} [\Delta_x] - \mathbb{E}[\tilde{\Delta}_x] \| = 0, \\
           \text{(ii) } \| \mathbb{D} [\Delta_x] - \mathbb{D}[\tilde{\Delta}_x] \| = \mathcal{O}(\bar{\pi}^4), \\
           \text{(iii) } \| \mathbb{E} [\Delta_x^{\otimes 3}] - \mathbb{E}[\tilde{\Delta}_x^{\otimes 3}] \| = \mathcal{O}(\bar{\pi}^3), \\
           \text{(iv) } | \Delta_y - \tilde{\Delta}_y | = \mathcal{O}(\bar{\pi}^5).
        \end{array}
    \end{equation*}
\end{lemma}
\begin{proof}[Proof of Lemma \ref{lem:one_step_dg}]
    For the exact dynamics of $\mathbf{G}_{lq}^{aux}$ under the Zero-Order Hold (ZOH) pattern, the control $z(t)$ is constant over $[t_0, t_1)$. The exact state transition is $x(t) = \Phi(t)x_0 + \Psi(t)z(t_0)$, where $\Phi(t) = e^{A(t-t_0)}$ and $\Psi(t) = \int_{t_0}^t e^{A(t-s)} B ds$. 
    By definition of the projection map, $\tilde{\Gamma}(t_0) = \mathbb{E}[z(t_0)]$. Furthermore, owing to the conditional independence of the players' sampling mechanisms (Assumption \ref{ass:independence}), the cross-covariance is zero, yielding $\tilde{\Sigma}(t_0) = \mathbb{D}[z(t_0)] = \diag(\tilde{\Sigma}_1(t_0), \tilde{\Sigma}_2(t_0))$.

    \underline{\textit{Proof of (i):}} Taking the expectation of the exact state yields $\mathbb{E}[x(t_1)] = \Phi(t_1)x_0 + \Psi(t_1)\tilde{\Gamma}(t_0)$. Due to the linearity of the surrogate drift and $\mathbb{E}[dW_t]=0$, the surrogate state expectation strictly evaluates to $\mathbb{E}[\tilde{x}(t_1)] = \Phi(t_1)x_0 + \Psi(t_1)\tilde{\Gamma}(t_0)$. Thus, the first moment difference is identically $0$.

    \underline{\textit{Proof of (ii):}} The exact covariance of the ZOH dynamics is governed by the matrix $\Psi(t)$:
    \begin{equation*}
        \mathbb{D}[x(t_1)] = \Psi(t_1) \tilde{\Sigma}(t_0) \Psi(t_1)^\top.
    \end{equation*}
    Using the Taylor expansion of the state transition integral $\Psi(t_1) = \int_{t_0}^{t_1} e^{A(t_1-s)} B ds = \delta B + \frac{\delta^2}{2}AB + \mathcal{O}(\delta^3)$, we obtain:
    \begin{equation}\label{eq:exact_cov}
        \mathbb{D}[x(t_1)] = \delta^2 B\tilde{\Sigma}B^\top + \frac{\delta^3}{2}\left(AB\tilde{\Sigma}B^\top \!+\! B\tilde{\Sigma}B^\top A^\top\right) + \mathcal{O}(\delta^4).
    \end{equation}
    Concurrently, applying the Itô isometry to the diffusion term of $\mathbf{\tilde{G}}_{lq}^{aux}$, the surrogate covariance is:
    \begin{equation*}
            \mathbb{D}[\tilde{x}(t_1)] = \delta \int_{t_0}^{t_1} e^{A(t_1-s)} B\tilde{\Sigma}(t_0)B^\top e^{A^\top(t_1-s)} ds.
    \end{equation*}
    Expanding the integrand yields:
    \begin{equation}\label{eq:approx_cov}
        \mathbb{D}[\tilde{x}(t_1)] = \delta^2 B\tilde{\Sigma}B^\top + \frac{\delta^3}{2}\left(AB\tilde{\Sigma}B^\top \!+\! B\tilde{\Sigma}B^\top A^\top\right) + \mathcal{O}(\delta^4).
    \end{equation}
    Remarkably, owing to the global linearity of the dynamics and the specific $\delta^{\frac{1}{2}}$ scaling design of the diffusion term, the second and third-order Taylor terms in \eqref{eq:exact_cov} and \eqref{eq:approx_cov} structurally cancel each other out. Thus, $\|\mathbb{D}[x(t_1)] - \mathbb{D}[\tilde{x}(t_1)]\| = \mathcal{O}(\delta^4) = \mathcal{O}(\bar{\pi}^4)$.

    \underline{\textit{Proof of (iii):}} In the exact ZOH formulation, the random control input $z(t_0)$ is sampled once and held constant over the interval $[t_0, t_1)$, meaning the exact state increment $\Delta_x$ is entirely devoid of intra-interval stochastic driving processes (i.e., no Brownian motion). Thus, the exact increment rigorously evaluates to $\Delta_x = \delta(Ax_0 + Bz(t_0)) + \mathcal{O}(\delta^2)$. Taking the third moment strictly yields $\mathbb{E}[\Delta_x^{\otimes 3}] = \delta^3 \mathbb{E}[(Ax_0 + Bz(t_0))^{\otimes 3}] + \mathcal{O}(\delta^4) = \mathcal{O}(\delta^3)$. 
    
    For the surrogate SDG, the increment comprises a drift component $D \sim \mathcal{O}(\delta)$ and a pure diffusion component $\kappa$. As established in (ii), the variance of the diffusion component scales as $\mathbb{E}[\kappa^{\otimes 2}] = \mathcal{O}(\delta^2)$. Since $\kappa$ is an Itô integral of a deterministic matrix, it is strictly Gaussian, implying its odd moments identically vanish (i.e., $\mathbb{E}[\kappa] = 0$ and $\mathbb{E}[\kappa^{\otimes 3}] = 0$). 
    Expanding the third moment $\mathbb{E}[(D+\kappa)^{\otimes 3}]$, the only non-zero terms are the pure drift $D^{\otimes 3} \sim \mathcal{O}(\delta^3)$ and the cross-terms involving one drift and two diffusion components. These cross-terms evaluate to permutations of $D \otimes \mathbb{E}[\kappa^{\otimes 2}]$, which specifically scale as $\mathcal{O}(\delta) \cdot \mathcal{O}(\delta^2) = \mathcal{O}(\delta^3)$. Thus, both exact and surrogate third moments inherently reside in $\mathcal{O}(\delta^3)$, bounding their difference rigidly at $\mathcal{O}(\bar{\pi}^3)$.

    \underline{\textit{Proof of (iv):}} The expected control cost terms inside the integral match identically because $\mathbb{E}[z^\top R z] = \tilde{\Gamma}^\top R \tilde{\Gamma} + \tr(\tilde{\Sigma} R)$. The accumulated state cost difference over $[t_0, t_1)$ solely depends on the state covariance difference established in (ii):
    \begin{equation*}
        \begin{aligned}
            \Delta_y - \tilde{\Delta}_y &= \int_{t_0}^{t_1} \tr\left\{ Q \left( \mathbb{D}[x(t)] - \mathbb{D}[\tilde{x}(t)] \right) \right\} dt \\
            &= \tr\left\{ Q \int_{t_0}^{t_1} \mathcal{O}((t-t_0)^4) dt \right\} = \mathcal{O}(\delta^5) = \mathcal{O}(\bar{\pi}^5).
        \end{aligned}
    \end{equation*}
    The local one-step proof is thus completed.
\end{proof}

According to Milshtein's theorem \cite[Theorem 2]{milshteinWeakApproximationSolutions1986a}, Lemma \ref{lem:one_step_dg} guarantees that the local weak approximation errors for the first three moments are strictly bounded by $\mathcal{O}(\bar{\pi}^3)$. Accumulating this $\mathcal{O}(\bar{\pi}^3)$ error per step over $N = (T-t_0)/\bar{\pi}$ intervals yields a global approximation error of $\mathcal{O}(\bar{\pi}^2)$. 

Consequently, the continuous state trajectory and the accumulated cost of the designed surrogate SDG $\mathbf{\tilde{G}}_{lq}^{aux}$ under pure strategies constitute a strict order $2$ weak approximation of the original game $\mathbf{G}_{lq}^{aux}$ under mixed strategies. Reverting the auxiliary variables to the original cost functional yields $|J(t_0, x_0, \alpha, \beta) - \tilde{J}(t_0, x_0, \phi_1(\alpha), \phi_2(\beta))| = \mathcal{O}(\bar{\pi}^2)$, which completes the proof.
\end{proof}

\section{Proof of Theorem \ref{thm:dg-lq-solution}} \label{app:thm:dg-lq-solution}
\begin{proof}
    For brevity, we drop the time argument $t$ and denote $\tilde{V}_p$ as $\tilde{V}$. The HJI equation for the continuous-time surrogate SDG is formulated as:
    \begin{equation}\label{eq:HJI_original}
        -\partial_t \tilde{V} = \inf_{\tilde{u} \in \tilde{\mathcal{C}}_1} \sup_{\tilde{v} \in \tilde{\mathcal{C}}_2} \left\{ \tilde{h} + (\partial_x \tilde{V})^\top \tilde{f} + \frac{1}{2} \tr\big( \tilde{\Sigma}_{diff} \partial_x^2 \tilde{V} \big) \right\},
    \end{equation}
    where the equivalent drift is $\tilde{f} = Ax + B_1\tilde{\Gamma}_1 + B_2\tilde{\Gamma}_2$ and the equivalent diffusion covariance is explicitly evaluated as $\tilde{\Sigma}_{diff} = \delta_\pi(t) \left(B_1 \tilde{\Sigma}_1 B_1^\top + B_2 \tilde{\Sigma}_2 B_2^\top\right)$.
    Substituting the ansatz $\tilde{V}(t,x)=x^\top P(t)x$ into the HJI equation \eqref{eq:HJI_original} with drift $\tilde{f}$ and diffusion $\tilde{\Sigma}_{diff}$, we observe that the min-max optimization decouples into two independent subproblems:
    \begin{equation}
        -\dot{P} = A^\top P + PA + Q + \inf_{\tilde{u}} \tilde{H}_1 + \sup_{\tilde{v}} \tilde{H}_2,
    \end{equation}
    where the local decoupled Hamiltonians are defined as:
    \begin{align*}
        \tilde{H}_1 &= 2x^\top P B_1 \tilde{\Gamma}_1 + \tilde{\Gamma}_1^\top R_1 \tilde{\Gamma}_1 + \tr\left( (\delta_\pi B_1^\top P B_1 + R_1) \tilde{\Sigma}_1 \right), \\
        \tilde{H}_2 &= 2x^\top P B_2 \tilde{\Gamma}_2 - \tilde{\Gamma}_2^\top R_2 \tilde{\Gamma}_2 + \tr\left( (\delta_\pi B_2^\top P B_2 - R_2) \tilde{\Sigma}_2 \right).
    \end{align*}

    \textbf{Step 1: Maximizer (Player 2).} We solve $\sup_{\tilde{\Gamma}_2, \tilde{\Sigma}_2 \succeq \mathbf{0}} \tilde{H}_2$ subject to $\|\tilde{\Gamma}_2\|^2 + \tr(\tilde{\Sigma}_2) \le \gamma_2^2 \|x\|^2$. Notice that the objective function is concave (strictly concave in $\tilde{\Gamma}_2$ and linear in $\tilde{\Sigma}_2$) and the constraints are convex. For any $x \neq 0$, Slater's condition is strictly satisfied (e.g., by selecting an interior point $\tilde{\Gamma}_2 = \mathbf{0}$ and $\tilde{\Sigma}_2 = \epsilon I \succ \mathbf{0}$), which mathematically guarantees strong duality. Thus, the Karush-Kuhn-Tucker (KKT) conditions are both necessary and sufficient for global optimality. Formulating the Lagrangian $\mathcal{L}_2 = 2x^\top P B_2 \tilde{\Gamma}_2 - \tilde{\Gamma}_2^\top (R_2 + \lambda I) \tilde{\Gamma}_2 + \tr((Q_{2,eq} - \lambda I)\tilde{\Sigma}_2) + \lambda \gamma_2^2 \|x\|^2$, we derive:
    
    (i) \textit{Dual Monotonicity and Optimal Mean:} For the supremum of $\mathcal{L}_2$ to be bounded, dual feasibility strictly requires $Q_{2,eq} - \lambda I \preceq \mathbf{0}$, hence $\lambda \ge \max(0, \lambda_{\max}(Q_{2,eq})) \triangleq \lambda_2(t)$. For any strictly feasible $\lambda > \lambda_2(t)$, the matrix $Q_{2,eq} - \lambda I$ is strictly negative definite, enforcing $\tilde{\Sigma}_2 = \mathbf{0}$, and the optimal unconstrained mean is $\tilde{\Gamma}_2^*(\lambda) = (R_2 + \lambda I)^{-1} B_2^\top P x$. Substituting these yields the dual function $g(\lambda) = x^\top P B_2 (R_2 + \lambda I)^{-1} B_2^\top P x + \lambda \gamma_2^2 \|x\|^2$. Its derivative evaluates to:
    $g'(\lambda) = \gamma_2^2 \|x\|^2 - \|\tilde{\Gamma}_2^*(\lambda)\|^2.$
    Under Assumption \ref{assump:actuation_capacity}, the matrix inequality $\gamma_2^2 I - P B_2(R_2+\lambda_2 I)^{-2}B_2^\top P \succeq \mathbf{0}$ inherently holds, mathematically guaranteeing $g'(\lambda_2(t)) \ge 0$. Furthermore, since the matrix inverse $(R_2 + \lambda I)^{-2}$ strictly decreases as $\lambda$ increases, the term $\|\tilde{\Gamma}_2^*(\lambda)\|^2$ strictly decreases, yielding $g'(\lambda) > 0$ for all $\lambda > \lambda_2(t)$. Because $g(\lambda)$ is strictly increasing on its feasible domain $[\lambda_2(t), \infty)$, the infimum of the dual problem $\min_{\lambda} g(\lambda)$ is uniquely attained at the left boundary $\lambda^* = \lambda_2(t)$. The optimal mean is therefore uniquely determined as $\tilde{\Gamma}_2^*(t) = (R_2 + \lambda_2(t) I)^{-1} B_2^\top P x \triangleq K_2(t)x$.

    (ii) \textit{Complementary Slackness and Optimal Variance:} The structure of the optimal variance strictly bifurcates based on $\lambda_{\max}(Q_{2,eq})$:
    \begin{itemize}
        \item If $\lambda_{\max}(Q_{2,eq}) \le 0$, the multiplier is $\lambda^* = 0$ and $Q_{2,eq} \preceq \mathbf{0}$. To maximize the non-positive trace term $\tr(Q_{2,eq}\tilde{\Sigma}_2)$, Player 2 optimally avoids any variance penalty by assigning $\tilde{\Sigma}_2^* = \mathbf{0}$.
        \item If $\lambda_{\max}(Q_{2,eq}) > 0$, the optimal multiplier is strictly positive ($\lambda^* = \lambda_{\max}(Q_{2,eq}) > 0$). The matrix $M \triangleq Q_{2,eq} - \lambda^* I \preceq \mathbf{0}$. To maximize $\tr(M\tilde{\Sigma}_2)$, the variance must align with the null space of $M$, resulting in a rank-1 form: $\tilde{\Sigma}_2^* = c \cdot v_2 v_2^\top$, where $v_2$ is the principal eigenvector of $Q_{2,eq}$. Crucially, the KKT complementary slackness requires $\lambda^* \big( \gamma_2^2 \|x\|^2 - \|\tilde{\Gamma}_2^*\|^2 - \tr(\tilde{\Sigma}_2^*) \big) = 0$. Since $\lambda^* > 0$, the energy constraint must be strictly active, uniquely dictating that the variance must absorb all residual energy: $c = \gamma_2^2 \|x\|^2 - \|K_2 x\|^2$. 
    \end{itemize}
    Combining both cases structurally yields the optimal variance augmented with the indicator function $\mathbb{I}_{\{\lambda_{\max}(Q_{2,eq}) > 0\}}$ as presented in the theorem.

    \textbf{Step 2: Minimizer (Player 1).} Given $Q_{1,eq} \succ \mathbf{0}$, the trace term $\tr(Q_{1,eq}\tilde{\Sigma}_1)$ is minimized at $\tilde{\Sigma}_1^* = \mathbf{0}$. Minimizing $\tilde{H}_1$ w.r.t $\tilde{\Gamma}_1$ yields $\tilde{\Gamma}_1^* = -R_1^{-1} B_1^\top P x$.

    \textbf{Step 3: Synthesis.} Substituting the optimal controls back into the HJI equation and identifying coefficients quadratic in $x$ yields the GRDE \eqref{eq:generalized_riccati}. The existence of $P(t)$ on $[t_0, T]$ is guaranteed by the capacity condition, completing the proof.
\end{proof}

\section{Proof of Lemma \ref{lem:smooth_replacement}}\label{app:lem:smooth_replacement}
\begin{proof}
    The bounds for $\epsilon_1, \epsilon_2,$ and $\epsilon_3$ are evaluated systematically through a four-step procedure: error dynamics analysis, state bound estimation, Grönwall's inequality application, and cost difference expansion. Let $e(t) = \tilde{x}^d(t) - \tilde{x}^*(t)$ be the state tracking error.

    \noindent\textbf{Part I: Bound for $\epsilon_1 = \mathcal{O}(\bar{\pi})$} 

    Evaluating $\epsilon_1$ corresponds to replacing $\tilde{\alpha}^*$ with $\operatorname{ZOH}_\pi[\tilde{\alpha}^*]$ against $\phi_2(\beta^*)$. 
    From Theorem \ref{thm:dg-lq-solution}, Player 1's analytical optimal variance strictly vanishes ($\tilde{\Lambda}_1^* \equiv \mathbf{0}$). Thus, the ZOH substitution only affects the drift. Let $\rho_1 = \max_t \|K_1(t)\|$.
    
    \underline{\textit{Step 1: Error dynamics.}} Applying Itô's formula to $\|e(t)\|^2$ for $t\in[t_{k-1},t_k)$:
    \begin{equation*}
    \begin{aligned}
       d\|e(t)\|^2 &= 2e(t)^\top \big\{A\tilde{x}^d(t) - A\tilde{x}^*(t) - B_1[K_1(t_{k-1})\tilde{x}^d(t_{k-1}) \\
       &\qquad - K_1(t)\tilde{x}^*(t)] +B_2[\tilde{\Gamma}_2^d(t_{k-1}) - \tilde{\Gamma}_2^*(t_{k-1})]\big\}dt\\
       &\quad +2e(t)^\top \delta_{\pi}^{\frac{1}{2}}(t)[0, B_2\tilde{\Lambda}_2^d(t_{k-1}) - B_2\tilde{\Lambda}_2^*(t_{k-1})]dW_t\\
       &\quad +\delta_{\pi}(t)\left\|B_2[\tilde{\Lambda}_2^d(t_{k-1}) - \tilde{\Lambda}_2^*(t_{k-1})]\right\|_F^2 dt.
    \end{aligned}
    \end{equation*}

    \underline{\textit{Step 2: States upper bound.}} Given that the opponent's strategy $\phi_2(\beta^*)$ satisfies linear growth with parameter $\kappa_2$ (note that the linear growth property of the optimal strategies is a direct structural consequence of the state-dependent energy constraints imposed in Definition 1), applying the concavity of the $2$-norm and Grönwall’s inequality guarantees uniformly bounded second moments. Define $M^* = \|A\|+\rho_1\|B_1\|+\kappa_2\|B_2\|$. It rigorously follows that $\mathbb{E}\|\tilde{x}^*(t)\|^2 \leq M^*_T$ and $\mathbb{E}\|\tilde{x}^d(t)\|^2 \le M_T^d$ for some positive constants $M_T^*$ and $M_T^d$.

    \underline{\textit{Step 3: Error upper bound.}} By expanding the drift difference into a cross-term $\mathbb{E}\|B_1K_1(t_{k-1})[\tilde{x}^d(t_{k-1})-\tilde{x}^*(t)]\|$ and a Lipschitz term $\mathbb{E}\|B_1[K_1(t_{k-1})-K_1(t)]\tilde{x}^*(t)\|$, we bound the difference strictly by $\mathcal{O}(\bar{\pi})$. Because $\phi_2(\beta^*)$ is globally Lipschitz (Assumption \ref{assump:lipschitz_mixed}), there exist constants $a_1, a_2, b_1, b_2 > 0$ such that the expected drift difference is bounded by $a_1\mathbb{E}\|e(t)\| + b_1\bar\pi$, and the Frobenius norm difference satisfies $\mathbb{E}\|B_2[\tilde{\Lambda}_2^d - \tilde{\Lambda}_2^*]\|_F \leq a_2\mathbb{E}\|e(t)\| + b_2\bar\pi$.
    Taking expectations on $d\|e(t)\|^2$:
    \begin{equation*}
        \frac{d}{dt}\mathbb{E}\|e(t)\|^2 \leq 2(\|A\|+a_1^2+a_2^2\bar\pi)\mathbb{E}\|e(t)\|^2 +2b_1^2\bar\pi^2 + 2b_2^2\bar\pi^3.
    \end{equation*}
    Applying Grönwall’s inequality yields the mean square state error $\mathbb{E}\|e(t)\|^2 = \mathcal{O}(\bar{\pi}^2)$.

    \underline{\textit{Step 4: Cost difference.}} Substituting $\mathbb{E}\|e(t)\|^2 = \mathcal{O}(\bar{\pi}^2)$ into the quadratic cost functional, the state penalty deviation is bounded via the Cauchy-Schwarz inequality: $\mathbb{E}[\tilde{x}^{*\top} Q \tilde{x}^* - \tilde{x}^{d\top} Q \tilde{x}^d] \le 2\|Q\| \sqrt{\max(M_T^d, M_T^*)} \sqrt{\mathbb{E}\|e(t)\|^2} = \mathcal{O}(\bar{\pi})$. The linear-quadratic control cost deviations structurally follow $\mathcal{O}(\bar{\pi})$. Thus, $\epsilon_1 = \mathcal{O}(\bar{\pi})$.

    \vspace{0.5em}
    \noindent\textbf{Part II: Bound for $\epsilon_2 = \mathcal{O}(\bar{\pi}^{\frac{1}{2}})$} 

    Evaluating $\epsilon_2$ replaces $\tilde{\beta}^*$ with $\operatorname{ZOH}_\pi[\tilde{\beta}^*]$ against $\phi_1(\alpha^*)$. Because $\tilde{\beta}^*$ is the exact optimal response (Theorem \ref{thm:dg-lq-solution}), its variance involves the indicator function $\mathbb{I}_{\{\lambda_{\max}(Q_{2,eq}) > 0\}}$, which introduces finite non-Lipschitz jump discontinuities.
    
    \underline{\textit{Step 1 \& 2:}} The error dynamics and state bounding follow the symmetric structure as Part I, guaranteeing bounded second moments $M_T^*$ and $M_T^d$.

    \underline{\textit{Step 3: Error upper bound.}} Crucially, although the non-Lipschitz indicator function precludes standard Lipschitz bounding techniques, the structural state-dependent energy constraints universally truncate the maximum possible variance injection. By applying the parallelogram inequality and the submultiplicativity of the Frobenius norm, we have:
    \begin{equation*}
        \begin{aligned}
            &\left\| B_2\tilde{\Lambda}_2^d(t_{k-1})-B_2\tilde{\Lambda}_2^*(t_{k-1})\right\|_F^2 \\
            &\leq 2\|B_2\|^2 \left( \|\tilde{\Lambda}_2^d(t_{k-1})\|_F^2 + \|\tilde{\Lambda}_2^*(t_{k-1})\|_F^2 \right) \\
            &= 2\|B_2\|^2 \left( \tr(\tilde{\Sigma}_2^d(t_{k-1})) + \tr(\tilde{\Sigma}_2^*(t_{k-1})) \right).
        \end{aligned}
    \end{equation*}
    Because the traces of the optimal variances are explicitly bounded by the state-dependent energy capacity $\gamma_2^2 \|\tilde{x}\|^2$, this dynamically evaluates to:
    \begin{equation*}
        \begin{aligned}
            &\left\| B_2\tilde{\Lambda}_2^d(t_{k-1})-B_2\tilde{\Lambda}_2^*(t_{k-1})\right\|_F^2 \\
            &\leq 2\|B_2\|^2 \gamma_2^2 \left( \|\tilde{x}^d(t_{k-1})\|^2 + \|\tilde{x}^*(t_{k-1})\|^2 \right).
        \end{aligned}
    \end{equation*}
    Based on the uniform state bounds established in Step 2, there exists a constant $C_x > 0$ such that $\mathbb{E}\|\tilde{x}^d(t_{k-1})\|^2 + \mathbb{E}\|\tilde{x}^*(t_{k-1})\|^2 \le C_x$.
    Taking expectations on $d\|e(t)\|^2$ injects this bounded local variance jump scaled by $\delta_\pi(t) \sim \mathcal{O}(\bar{\pi})$:
    \begin{equation*}
        \frac{d}{dt}\mathbb{E}\|e(t)\|^2 \leq 2(\|A\|+a_1^2)\mathbb{E}\|e(t)\|^2 + 2b_1^2\bar\pi^2 + 2\|B_2\|^2 \gamma_2^2 C_x \bar\pi.
    \end{equation*}
    Due to the $\mathcal{O}(\bar{\pi})$ constant term, applying Grönwall’s inequality degrades the accumulated state error to $\mathbb{E}\|e(t)\|^2 = \mathcal{O}(\bar{\pi})$.

    \underline{\textit{Step 4: Cost difference.}} Substituting $\mathbb{E}\|e(t)\|^2 = \mathcal{O}(\bar{\pi})$ into the cost functional, the state penalty expands via Cauchy-Schwarz as $\mathbb{E}[\tilde{x}^{*\top} Q \tilde{x}^* - \tilde{x}^{d\top} Q \tilde{x}^d] \le C_Q \sqrt{\mathbb{E}\|e(t)\|^2} = \mathcal{O}(\bar{\pi}^{\frac{1}{2}})$. 
    Crucially, because the optimal variance $\tilde{\Sigma}_2^*(t, \tilde{x})$ operates as a state-feedback mechanism, its deviation must be rigorously decoupled into a state mismatch error and a time discretization error:
    \begin{equation*}
        \begin{aligned}
            \tilde{\Sigma}_2^d(t) - \tilde{\Sigma}_2^*(t) =& \underbrace{\tilde{\Sigma}_2^*(t_{k-1}, \tilde{x}^d(t_{k-1})) - \tilde{\Sigma}_2^*(t_{k-1}, \tilde{x}^*(t_{k-1}))}_{\text{State mismatch error}} \\
            &+ \underbrace{\tilde{\Sigma}_2^*(t_{k-1}, \tilde{x}^*(t_{k-1})) - \tilde{\Sigma}_2^*(t, \tilde{x}^*(t))}_{\text{Time discretization error}}.
        \end{aligned}
    \end{equation*}
    For the state mismatch error, because $\tilde{\Sigma}_2^*(t, \tilde{x})$ is quadratic in $\tilde{x}$, Cauchy-Schwarz bounds the expected deviation by $\mathcal{O}(\sqrt{\mathbb{E}\|e(t_{k-1})\|^2}) = \mathcal{O}(\bar{\pi}^{\frac{1}{2}})$. 

    For the time discretization error, note that the indicator function $\mathbb{I}_{\{\lambda_{\max}(Q_{2,eq}(t)) > 0\}}$ relies exclusively on the deterministic Riccati solution $P(t)$, thus introducing only a finite number of deterministic jumps. Furthermore, the expected quadratic state term $\mathbb{E}[\tilde{x}^*(t)\tilde{x}^*(t)^\top]$ defines the state covariance matrix, which evolves according to a smooth deterministic ordinary differential equation. Consequently, the expectation $\mathbb{E}[\tr(\tilde{\Sigma}_2^*(t, \tilde{x}^*(t)) R_2)]$ intrinsically forms a piecewise continuously differentiable function. As a function of bounded variation, its left Riemann sum approximation error mathematically guarantees an $\mathcal{O}(\bar{\pi})$ bound.
    
    Summing these composite mechanisms, the $\mathcal{O}(\bar{\pi}^{\frac{1}{2}})$ tracking error strictly dominates the variance cost difference, yielding $\epsilon_2 = \mathcal{O}(\bar{\pi}^{\frac{1}{2}})$.

    \vspace{0.5em}
    \noindent\textbf{Part III: Bound for $\epsilon_3 = \mathcal{O}(\bar{\pi}^{\frac{1}{2}})$} 

    Evaluating $\epsilon_3$ involves ZOH substitutions for both continuous optimal strategies. The error dynamics inherit both the Lipschitz drift delays from Player 1 and the non-Lipschitz variance jumps from Player 2. Driven by the dominant diffusion discontinuity identified in Part II, Grönwall's inequality yields an identical mean square tracking bound $\mathbb{E}\|e(t)\|^2 = \mathcal{O}(\bar{\pi})$. By expanding the state and control cost functionals via the Cauchy-Schwarz inequality, the quadratic discrepancies are structurally governed by $\mathcal{O}(\sqrt{\mathbb{E}\|e(t)\|^2}) + \mathcal{O}(\bar{\pi}) = \mathcal{O}(\bar{\pi}^{\frac{1}{2}})$. Thus, $\epsilon_3 = \mathcal{O}(\bar{\pi}^{\frac{1}{2}})$.
\end{proof}

\section{Proof of Lemma \ref{lem:best_response_deviation}}\label{app:lem:best_response_deviation}
\begin{proof}
    We rigorously establish the deviation bounds for $\epsilon_4$ and $\epsilon_5$ by analyzing the closed-loop continuous-time optimal tracking best responses against the opponent's strategy mismatch.

    \noindent\textbf{Part I: Bound for $\epsilon_4 = \mathcal{O}(\bar{\pi}^{\frac{1}{2}})$} 

    Evaluating $\epsilon_4$ quantifies $|\tilde{J}(\tilde{\alpha}^{br}, \operatorname{ZOH}_\pi[\tilde{\beta}^*]) - \tilde{J}(\tilde{\alpha}^{br}, \tilde{\beta}^*)|$. Because the opponent's strategy $\operatorname{ZOH}_\pi[\tilde{\beta}^*]$ preserves linear growth with respect to the state, the continuous optimal best response $\tilde{\alpha}^{br}$ amounts to solving an LQ tracking problem subject to piecewise-constant external signals. Standard linear-quadratic tracking theory ensures that $\tilde{\alpha}^{br}$ consists of a smooth linear state feedback term $\tilde{\Gamma}_1^{br}(x) = -K_1^{br}(t)x + d(t)$ and zero variance $\tilde{\Lambda}_1^{br} = \mathbf{0}$, meaning $\tilde{\alpha}^{br}$ is uniformly globally Lipschitz.

    Let $e(t) = \tilde{x}^d(t) - \tilde{x}^*(t)$ denote the state error when evaluating the fixed tracking response $\tilde{\alpha}^{br}$ against $\operatorname{ZOH}_\pi[\tilde{\beta}^*]$ versus $\tilde{\beta}^*$. Because $\tilde{\beta}^*$ contains the non-Lipschitz variance indicator jumps (Theorem \ref{thm:dg-lq-solution}), the diffusion error over $[t_{k-1}, t_k)$ mimics the exact structural discrepancy detailed in Part II of Lemma \ref{lem:smooth_replacement}. Namely, $\mathbb{E}\|B_2(\tilde{\Lambda}_2^d - \tilde{\Lambda}_2^*)\|_F^2 \le \mathcal{O}(1)$. 
    Applying Itô's formula and Grönwall's inequality yields a mean square state error of $\mathbb{E}\|e(t)\|^2 = \mathcal{O}(\bar{\pi})$.
    
    Bounding the quadratic state penalty difference exclusively through the Cauchy-Schwarz inequality provides $\mathbb{E}[\tilde{x}^{*\top} Q \tilde{x}^* - \tilde{x}^{d\top} Q \tilde{x}^d] \le C \sqrt{\mathbb{E}\|e(t)\|^2} = \mathcal{O}(\bar{\pi}^{\frac{1}{2}})$. Therefore, $\epsilon_4 = \mathcal{O}(\bar{\pi}^{\frac{1}{2}})$.

    \vspace{0.5em}
    \noindent\textbf{Part II: Bound for $\epsilon_5 = \mathcal{O}(\bar{\pi})$} 

    Symmetrically, evaluating $\epsilon_5$ requires to quantify $|\tilde{J}(\operatorname{ZOH}_\pi[\tilde{\alpha}^*], \tilde{\beta}^{br}) - \tilde{J}(\tilde{\alpha}^*, \tilde{\beta}^{br})|$. The optimal tracking response $\tilde{\beta}^{br}$ is naturally globally Lipschitz.

    Crucially, the deviation is driven exclusively by the ZOH mechanism acting on Player 1's analytical strategy $\tilde{\alpha}^*$. Since Player 1's exact optimal variance is structurally zero ($\tilde{\Lambda}_1^* \equiv \mathbf{0}$), $\operatorname{ZOH}_\pi[\tilde{\alpha}^*]$ does not introduce any non-Lipschitz jump discrepancies in the diffusion matrix (i.e., $\mathbb{E}\|B_1(\tilde{\Lambda}_1^d(t_{k-1}) - \tilde{\Lambda}_1^*(t_{k-1}))\|_F^2 \equiv 0$). 
    The state error $e(t)$ is propelled purely by the Lipschitz continuous drift delays, perfectly analogous to Part I of Lemma \ref{lem:smooth_replacement}. Consequently, Grönwall's inequality preserves the higher-order tracking bound, yielding $\mathbb{E}\|e(t)\|^2 = \mathcal{O}(\bar{\pi}^2)$.
    Substituting this refined error back into the cost functional immediately bounds the state penalty difference via Cauchy-Schwarz by $\mathcal{O}(\sqrt{\mathbb{E}\|e(t)\|^2}) = \mathcal{O}(\bar{\pi})$. The control cost differences follow similarly, ensuring $\epsilon_5 = \mathcal{O}(\bar{\pi})$.
\end{proof}

\end{appendices}

\begin{IEEEbiographynophoto}{Tao Xu}
	(S'22) received a B.S. degree in the School of Mathematical Sciences from Shanghai Jiao Tong University (SJTU), Shanghai, China. He is currently working toward the Ph.D. degree with the Department of Automation, SJTU. He is a member of Intelligent of Wireless Networking and Cooperative Control Group. His research interests include probabilistic prediction, distributionally robust optimization, dynamic games, and robotics.
\end{IEEEbiographynophoto}

\begin{IEEEbiographynophoto}{Wang Xi}
	(S'24) received the B.S. degree in the School of Electronic Information and Electrical Engineering from Shanghai Jiao Tong University (SJTU), Shanghai, China. Since 2023, he has been with the Department of Automation, SJTU. He is a member of Intelligent of Wireless Networking and Cooperative Control Group. His current research interests include differential games and its applications in mechatronic systems.
\end{IEEEbiographynophoto}

\begin{IEEEbiographynophoto}{Jianping He}
	(SM'19) is currently a full professor in the Department of Automation at Shanghai Jiao Tong University. He received the Ph.D. degree in control science and engineering from Zhejiang University, Hangzhou, China, in 2013, and had been a research fellow in the Department of Electrical and Computer Engineering at University of Victoria, Canada, from Dec. 2013 to Mar. 2017. His research interests mainly include the distributed learning, control and optimization, security and privacy in network systems.

	Dr. He serves as an Associate Editor for IEEE Trans. Control of Network Systems, IEEE Trans. on Vehicular Technology, IEEE Open Journal of Vehicular Technology, and KSII Trans. Internet and Information Systems. He was also a Guest Editor of IEEE TAC, International Journal of Robust and Nonlinear Control, etc. He was the winner of Outstanding Thesis Award, Chinese Association of Automation, 2015. He received the best paper award from IEEE WCSP'17, the best conference paper award from IEEE PESGM'17, and was a finalist for the best student paper award from IEEE ICCA'17, and the finalist best conference paper award from IEEE VTC20-FALL.
\end{IEEEbiographynophoto}

\end{document}